\documentclass[a4paper]{article}

\usepackage[margin=1in]{geometry} 

\usepackage{amsmath}
\usepackage{amsthm}
\usepackage{amssymb}
\usepackage[dvipsnames]{xcolor}
\usepackage{cite}
\usepackage{comment}
\usepackage{graphicx}
\usepackage{multirow}
\usepackage{mathtools}

\mathtoolsset{showonlyrefs}

\usepackage[utf8]{inputenc}
\usepackage{hyperref}
\hypersetup{
	unicode,
	pdfauthor={Author One, Author Two, Author Three},
	pdftitle={A simple article template},
	pdfsubject={A simple article template},
	pdfkeywords={article, template, simple},
	pdfproducer={LaTeX},
	pdfcreator={pdflatex}
}

\usepackage[sort&compress,numbers,square]{natbib}

\theoremstyle{plain}
\newtheorem{theorem}{Theorem}
\newtheorem{corollary}[theorem]{Corollary}
\newtheorem{lemma}[theorem]{Lemma}

\newtheorem{proposition}[theorem]{Proposition}

\theoremstyle{definition}
\newtheorem{definition}[theorem]{Definition}
\newtheorem{example}[theorem]{Example}
\newtheorem{remark}[theorem]{Remark}

\newcommand{\NN}{\mathbb{N}}

\newcommand{\RR}{\mathbb{R}}

\newcommand{\ds}{\displaystyle}
\newcommand{\dr}{\mathrm{d}}
\newcommand{\drdth}{\;\mathrm{d} r\;\mathrm{d}\theta}
\newcommand{\dxdy}{\;\mathrm{d} x\;\mathrm{d} y}

\newcommand{\innerprod}[2]{\langle #1 , #2 \rangle}

\usepackage[colorinlistoftodos,prependcaption,textsize=small]{todonotes}

\definecolor{greenpie}{rgb}{0.69, 0.95, 0.76}

\usepackage{algorithm, algpseudocode} 
\usepackage{mathrsfs} 

\usepackage{lipsum}

\title{Multiple Orthogonal Polynomials on the Ball and Radial Extensions}
\author{Lidia Fernández$^\star$ \and Juan Antonio Villegas$^\star$}

\date{
    \texttt{lidiafr@ugr.es},\qquad \texttt{jantoniovr@ugr.es}\\[2ex]%
	$^\star$Departamento de Matemática Aplicada and Instituto de Matemáticas (IMAG). Universidad de Granada \\[2ex]%
}

\begin{document}
	\maketitle
	
	\begin{abstract}
    
    A primary method for constructing orthogonal polynomials on the unit ball consists of combining a Jacobi-type radial component with a spherical harmonic angular part. Building upon this framework and using Jacobi-Pi\~neiro multiple orthogonal polynomials, this paper introduces Type~I and Type~II multiple orthogonal polynomials on the multidimensional ball. To demonstrate the practical utility of these definitions, we establish multivariate extensions of several fundamental results from univariate multiple orthogonality. Finally, we extend the construction to more general domains by introducing multiple orthogonality with respect to radial weights.
		

	\end{abstract}


\section{Introduction}\label{sec:intro}

Multiple orthogonal polynomials arise when the classical notion of orthogonality is extended to settings involving several measures instead of a single one. In this framework, a polynomial is required to satisfy a collection of orthogonality conditions, each associated with a different measure, and these conditions are organized according to a multi‑index. This naturally leads to two complementary formulations, usually referred to as Type~I and Type~II, which together capture the structure of the underlying system (see \cite{Ism05, MFVA16, VA20} for general settings in the univariate case). \medskip

The theory gained prominence as it became clear that these polynomials appear in a variety of contexts where several influences or sources interact simultaneously. They play a central role in Hermite–Padé approximation, where one seeks to approximate vector‑valued functions in a coordinated way \cite{VA06}. They also surface in problems from probability and mathematical physics, such as random matrix models with an external field or systems of non-intersecting paths \cite{BK04, DK07}. In these settings, the recurrence relations and analytic properties of the polynomials provide an effective means of describing the behavior of the models. \medskip

While univariate multiple orthogonality and classical multivariate orthogonality have long been, and continue to be, the focus of extensive research, work on multivariate multiple orthogonality is still in its early stages. In \cite{FV26}, the authors introduced a general framework for defining Type~I and Type~II bivariate multiple orthogonal polynomials on a two‑dimensional lattice, along with several extensions and illustrative examples. More recently, \cite{MRW25} examined the mixed‑type case, focusing on bivariate multiple orthogonal polynomials along the step‑line. In \cite{FFV26}, the authors further contributed to the field by developing a Rodrigues-type formula for multiple orthogonal polynomials defined on the simplex and, in addition, they established a connection with a Hermite–Padé type approximation of the kind introduced in \cite{Sor02}. \medskip

When studying systems of orthogonal polynomials on the unit ball, one of the most commonly used bases consists of polynomials obtained as the product of a Jacobi‑type polynomial in the radial variable and a spherical harmonic. One of the main objectives of this work is to build on this idea in order to construct Type~I and Type~II multiple orthogonal polynomials on the ball. To this end, we replace the radial component with Jacobi–Piñeiro multiple orthogonal polynomials while retaining the spherical harmonic. In this way, we show that the resulting family forms a basis of polynomials satisfying multiple orthogonality relations on the ball. Using this methodology, we also aim to extend orthogonal polynomials associated with general radial weights to the multiple orthogonal setting, providing an example illustrating how multiple Laguerre polynomials can be used to construct multivariate polynomials. In addition, we generalize existing results on multiple orthogonality to this multivariate framework.

\medskip

The structure of the paper is as follows. Section~\ref{sec:pre} provides essential preliminary tools, detailing properties of Multiple Orthogonal Polynomials, with a specific focus on the Jacobi-Pi\~neiro and multiple Laguerre polynomials, as well as an overview of spherical harmonics and orthogonal polynomials on the unit ball. In Section~\ref{sec:mop-ball}, we present extensions of orthogonal polynomials on the ball to the multiple orthogonal framework, explicitly describing the case for two variables (the disk). Afterwards, in Section~\ref{sec:NNR}, these newly developed constructions and their properties are applied to establish multivariate extensions of two fundamental results from univariate multiple orthogonality: the biorthogonality relation and Nearest Neighbor Relations. Finally, we study an extension to the multiple setting for orthogonal polynomials for radial weights in Section~\ref{sec:radial-weights}, introducing MOPs for multivariate Hermite Weights as an example.

\section{Preliminaries}\label{sec:pre}

With the aim of introducing an extension of MOPs in multi-dimensional domains, in this section we recall some tools and fix some notations: generalities about multiple orthogonal polynomials in one variable, focusing on the special case of Jacobi-Pi\~neiro polynomials, spherical harmonics, and orthogonal polynomials on the unit ball. 

\subsection{Multiple Orthogonal polynomials}

In the Orthogonal Polynomials setting, multiple orthogonality is a generalization of standard orthogonality where, rather than fixing a single measure and, in turn, an integral inner product $\innerprod{\cdot}{\cdot}_\omega$, several measures are provided, so that polynomials, named Multiple Orthogonal Polynomials (MOPs) satisfy orthogonality relations with respect to a system of measures and, in turn, a set of inner products  \cite{VA20, Ism05, Apt98}.

More precisely, consider a set of positive weight functions $\omega_1,\dots,\omega_r$ and their associated integral inner products
\begin{equation*}
    \innerprod{f}{g}_l := \innerprod{f}{g}_{\omega_l}=\int_\RR f(t)g(t)\omega_l(t)\dr t,
\end{equation*}
for $f,g\in \bigcap_{l=1}^r L^2(\RR;\omega_l)$. The orthogonality relations with respect to these inner products, as well as the degree of the polynomials, are determined by a multi-index $\vec n =(n_1,\dots,n_r)\in\NN^r_0$. We will denote the modulus or the norm of a vector $\vec v =(v_1,\dots,v_m)\in\mathbb R^m$ as $|\vec v|=|v_1|+\cdots +|v_m|$, so that $|\vec n| =n_1+\cdots+n_r$.\\

Fixing $\vec n\in\NN^r_0$, there are two ways of defining multiple orthogonal polynomials, known as Type~I and Type~II multiple orthogonal polynomials. Type~I MOPs are $r$ polynomials often represented in a vector $(A_{\vec n,1},\dots,A_{\vec n,r})$, such that for $l=1,\dots, r$ the degree of  $A_{\vec n,l}$ is at most $n_l-1$ and they satisfy 
\begin{equation}
\label{eq:type-i-orthogonality-1-var}    
\sum_{l=1}^r \innerprod{A_{\vec n,j}}{t^k}_l=\delta_{k,|\vec n|-1}, \qquad 0\leq k \leq |\vec n|-1.
\end{equation}

On the other hand, Type~II MOP is a polynomial (usually considered monic) $P_{\vec n}$ of degree $|\vec n|$ which is orthogonal to polynomials up to degree $n_l-1$ with respect to the $l$-th measure. This means
\begin{equation}
    \label{eq:type-ii-orthogonality-1-var}
    \innerprod{P_{\vec n}}{t^k}_l = 0, \quad 0\leq k\leq n_l-1, \quad \text{ for } l=1,\dots, r.
\end{equation}

These different ways of constructing multiple orthogonal polynomials are, indeed, equivalent for a fixed multi-index $\vec n$, which is said to be \emph{normal} if the existence of both Type~I and Type~II polynomials holds. If every multi-index is normal for a set of weights $\omega_1,\dots,\omega_r$ and the associated measures, this system of measures is called a \emph{perfect system}. 

Multiple orthogonality has several results extending some others of the standard orthogonal polynomials theory: existence and uniqueness theorems, Christoffel-Darboux formulas or Hermite-Padé Approximation, among others, see \citep[Section 23.1]{Ism05}, or \citep[Section 1]{VA20}. In particular, we will spotlight the Nearest Neighbour Recurrence Relations (NNRR), which extend the three-term recurrence relation satisfied by univariate orthogonal polynomial sequences.

From now on, assume that every mentioned multi-index is normal and let $\vec n,\vec m \in \NN^r$. The multi-indices $\vec n$ and $\vec m$ are said to be \emph{neighbours} if either $\vec n = \vec m + \vec e_l$ or $\vec n = \vec m - \vec e_l$ for any $l\in\{1,\dots,r\}$, where $\vec e_l = (0,\cdots,0, \overset{(l)}{1},0,\cdots,0)$ is a vector whose components are zero excepting the one in the $l$-th position, which is one. If we fix $\vec n=(n_1,\dots,n_r)\in\NN^r$, and consider $\vec m_{1},\dots,\vec m_r,\vec m_{r+1}=\vec n$ a path of neighbour multi-indices from $\vec m_1 =(n_1-1,\dots,n_r-1)$ to $\vec n$ and choose $\vec n+\vec e_k$ for any $k\in\{1,\dots,r\}$, the following relation holds \cite[Theorem 23.1.7]{Ism05}:
\begin{equation}
    \label{eq:NNRR-1-var}
    t P_{\vec n}(t) = P_{\vec n+\vec e_k}(t) + a_{\vec n,0} P_{\vec n}(t) + \sum_{l=1}^r a_{\vec n,l} P_{\vec m_l}(t).
\end{equation}
A different relation concerning the closest multi-indices $\vec n-\vec e_l$ for $l=1,\dots,r$ would be \cite[Section 3.2]{VA20}
\begin{equation}
    \label{eq:NNRR-1-var-2}
    t P_{\vec n}(t) = P_{\vec n+\vec e_k}(t) + b_{\vec n,0} P_{\vec n}(t) + \sum_{l=1}^r b_{\vec n,l} P_{\vec n-\vec e_l}(t).
\end{equation}

\begin{figure}[h]
    \centering
    \begin{tabular}{cccc}
         \includegraphics[height=3.5cm]{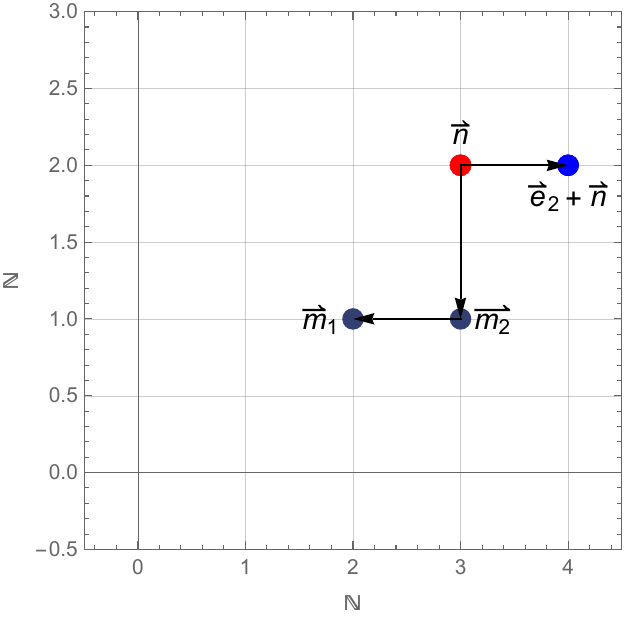} &
         \includegraphics[height=3.5cm]{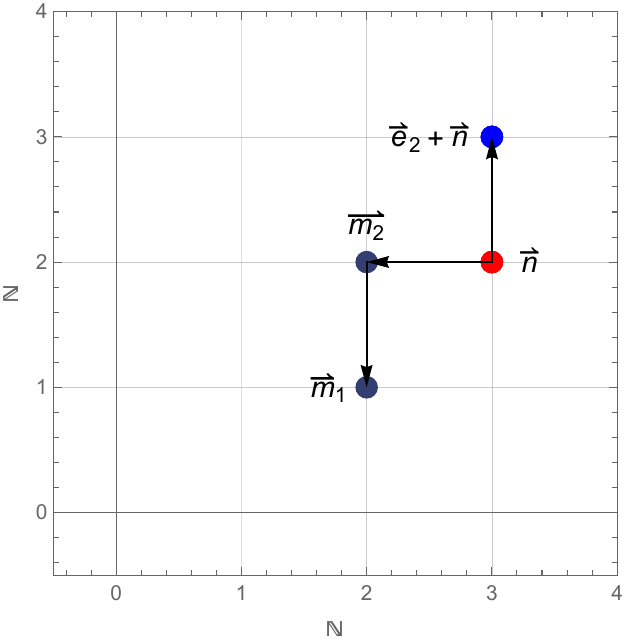} &
         \includegraphics[height=3.5cm]{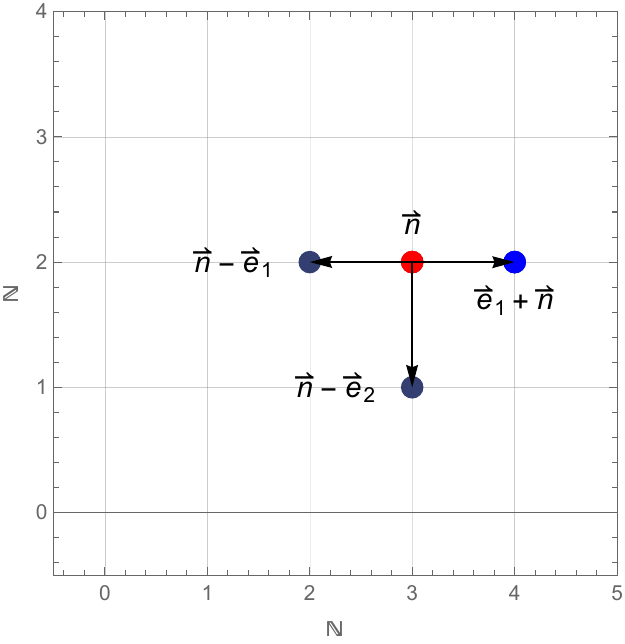} &
         \includegraphics[height=3.5cm]{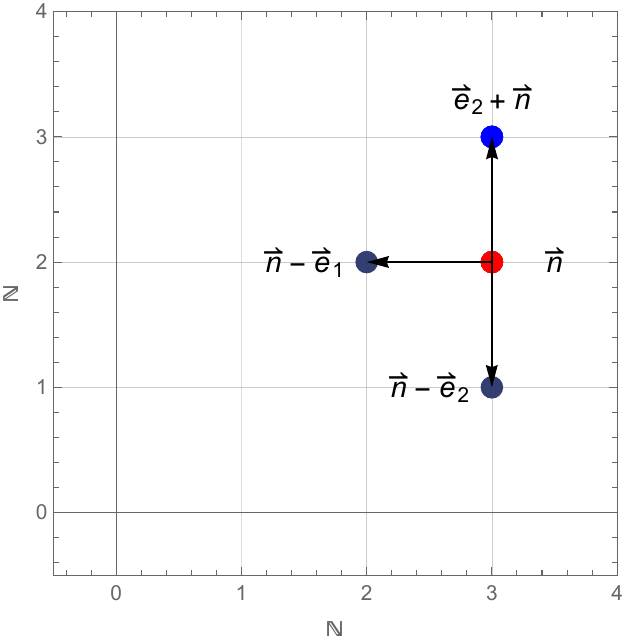} \\
         (a) & (b) & (c) & (d) \\
         \multicolumn{2}{c}{(a,b) Possible multi-indices involved in \eqref{eq:NNRR-1-var}} & 
         \multicolumn{2}{c}{(c,d) Possible multi-indices involved in \eqref{eq:NNRR-1-var-2}}
    \end{tabular}
    \caption{Plot of some choices of the multi-indices in the NNRR for $r=2$ and $\vec n =(3,2)$.}
    \label{fig:NNRR-1-var}
\end{figure}

See Figure~\ref{fig:NNRR-1-var} for a plot of the concerned multi-indices in both relations using 2 measures and $\vec n = (3,2)$. Besides, there exist analogous NNRR concerning Type~I MOP, see \cite[Theorem 23.1.9]{Ism05} and \cite[Section 3.2]{VA20}. 

\subsection{Jacobi-Pi\~neiro Multiple Orthogonal Polynomials}\label{subsec:JP}

In this document, we will center our attention on a specific system of measures which gives rise to the so-called \emph{Jacobi-Pi\~neiro multiple orthogonal polynomials}, a multiple orthogonal extension of the well-known Jacobi polynomials.

Most of the time, Jacobi-Pi\~neiro polynomials are introduced as those satisfying \eqref{eq:type-i-orthogonality-1-var} and \eqref{eq:type-ii-orthogonality-1-var} with the weights (see \cite[Section 3.7]{VA20})
$$
\tilde\omega_l(t):=t^{\alpha_l} (1-t)^{\gamma}, \qquad 0\leq t\leq 1,\quad  1\leq l\leq r.
$$
These polynomials were first introduced in \cite{Pin87} for the case $\gamma=0$. In this paper, we map the variable $t$ so that $t\longmapsto1-t$. With this change, we obtain different expressions -- but still computable using the standard ones -- that are more suitable for our concerns. 

Then, given $\alpha,\gamma_1,\dots,\gamma_p >-1$ such that $\gamma_i-\gamma_j\not\in\mathbb Z$ whenever $i\neq j$, consider the following weight functions supported on $[0,1]$: 
\begin{equation}
    \label{eq:JP-wights-1-var}
    \omega_l(t):=t^\alpha (1-t)^{\gamma_l}, \qquad 0\leq t\leq 1,\quad  1\leq l\leq r.
\end{equation}
The condition $\gamma_i-\gamma_j\not\in\mathbb Z$ whenever $i\neq j$ guaranties that this set of weights composes an AT-System, which is a particular case of a perfect system \cite[Section 23.1.2]{Ism05}. In this way, the existence of both types of MOPs is ensured for every $\vec n\in\mathbb N^r$. 

Denote $\vec\gamma=(\gamma_1,\dots,\gamma_r)$ and consider a multi-index $\vec n =(n_1,\dots,n_r)\in\NN_0^r$. The Type~II Jacobi-Pi\~neiro multiple orthogonal polynomials can be computed as \cite{ABVA03}
\begin{equation}\label{eq:JP-type-ii}
P^{(\alpha,\vec\gamma)}_{\vec n}(t)= \left(\frac 1 {\omega_r(t)}\dfrac{d^{n_r}}{dt^{n_r}}\omega_r(t)(1-t)^{n_r}\right) \cdots \left(\frac 1 {\omega_2(t)}\dfrac{d^{n_2}}{dt^{n_2}}\omega_2(t)(1-t)^{n_2}\right) \left(\frac 1 {\omega_1(t)}\dfrac{d^{n_1}}{dt^{n_1}}\omega_1(t)(1-t)^{n_1}\right)t^{|\vec n|}.
\end{equation}

These polynomials have degree $|\vec n|$ and are orthogonal to every polynomial of degree lower than $n_l$ with respect to $\omega_l$. More precisely, 
\begin{equation}\label{eq:multiple-orthogonality-JP-type-ii}
\langle P^{(\alpha,\vec\gamma)}_{\vec n}(t), t^k\rangle_l = \int_0^1  P^{(\alpha,\vec\gamma)}_{\vec n}(t)\, t^k \,t^{\alpha}(1-t)^{\gamma_l}\,\dr t = 0 \qquad\text{ if }0\leq k\leq n_l-1,\qquad 1\leq l\leq r.
\end{equation}

This representation in terms of a Rodrigues-type formula was introduced in \cite{ABVA03}, together with other expressions for Type~II MOP with respect to classical weights. Owing to their reduced complexity, Type~II MOPs have historically received greater attention than their Type~I counterparts. Consequently, explicit hypergeometric representations for Type~I Jacobi-Pi\~neiro polynomials have only recently emerged. Initial formulations for the two-measure case ($r=2$) were established in \cite{BDFM23} and \cite{BDFMAF23}, with subsequent generalizations for an arbitrary number of measures $r\geq 2$ provided in \cite{BDFM24}:
\begin{multline}\label{eq:JP-type-ii-type-i}
    A^{(\alpha,\vec\gamma)}_{\vec n,l}(t) = \dfrac{\prod_{k=1}^r(\gamma_k+\alpha+|\vec n|)_{n_k}}{(n_l-1)!\prod_{k=1,k\neq l}^{p}(\gamma_k-\gamma_l)_{n_k}}\dfrac{\Gamma(\gamma_l+\alpha+|\vec n|)}{\Gamma(\alpha+|\vec n|)\Gamma(\gamma_l+1)}\times \\ \times {_{r+1}}F_r\left( \begin{matrix} -n_l+1,\gamma_l+\alpha+|\vec n|,(\gamma_l+1)\vec 1_{r-1}-\vec\gamma^{(l)}-\vec n^{(l)} \\ \gamma_l+1,(\gamma_l+1)\vec 1_{r-1}-\vec\gamma^{(l)} \end{matrix} \,;\, 1-t \right),
\end{multline}
where we denote $\vec 1_{r-1}=(1,\dots,1)\in\mathbb R^{(r-1)}$ and given $\vec a\in\mathbb R^r$, $\vec a^{(l)}=(a_1,\dots,a_{l-1},a_{l+1},\dots,a_r)\in\RR^{r-1}$ is the vector $\vec a$ with the $l$-th component removed.

As mentioned in the previous section, the polynomials $A^{(\alpha,\vec\gamma)}_{\vec n,l}$ have degree at most $n_l-1$, $1\leq l\leq r$, and all together satisfy 
\begin{equation}\label{eq:multiple-orthogonality-JP-type-i}
    \sum_{l=1}^r \langle A^{(\alpha,\vec\gamma)}_{\vec n,l},t^k \rangle_l =\sum_{l=1}^r\int_0^1A^{(\alpha,\vec\gamma)}_{\vec n,l}\,t^k t^{\alpha}(1-t)^{\gamma_l}\,\dr t =\begin{cases}
        0 & \text{ if } 0\leq k\leq |\vec n|-2,\\
        1  & \text{ if } l= |\vec n|-1.
    \end{cases}
\end{equation}
Jacobi-Pi\~neiro polynomials have attracted significant attention within the field of multiple orthogonality since their introduction. Various representations in terms of hypergeometric functions were established in \citep{ABVA03, BCVA05}, reflecting their analytical versatility. Beyond their theoretical properties, they play a crucial role in several practical contexts, including rational approximation \cite{Apt98, VA20}, electrostatic partners \cite{MFOSL22}, number theory \cite{Fis04}, and even random walks \cite{BDFMAF23}, urn models \cite{GI21} and related stochastic processes \citep{BDFM25}, see also \cite{MFVA16}.

\subsection{Multiple Laguerre polynomials}\label{subsec:Laguerre-MOP}

Another commonly used family of multiple orthogonal polynomials is that extending classical Laguerre polynomials to this multiple orthogonal setting. Recall that, given $\alpha>-1$, Laguerre polynomials $\{L_n^{(\alpha)}\}_{n\geq 0}$ are those orthogonal with respect to the weight $t^\alpha e^{-t}$ on $[0,\infty)$. We may extend Laguerre polynomials to the multiple case by varying the parameter $\alpha$, considering $\alpha_1,\dots,\alpha_r>-1$, giving rise to multiple Laguerre polynomials \emph{of the first kind}, see \cite[Section 3.2]{ABVA03}, \cite[Section 23.4.1]{Ism05} or \cite[Section 3.6.1]{VA20}. However, we will work with a different extension, where the exponential decay at infinity is $e^{-t}$ is replaced by $e^{-c_l t}$ with parameters $c_1,\dots,c_r$, obtaining multiple Laguerre polynomials \emph{of the second kind} \cite[Section 3.1]{BK05}. 

More precisely, let $\alpha>-1$ and $c_1,\dots,c_r>0$ with $c_i\neq c_j$ whenever $i\neq j$, and define the weights \cite[Section 3.6.2]{VA20}
$$
\omega_l(t):= t^\alpha e^{-c_l t},\qquad t\geq 0, \quad 1\leq l\leq r.
$$
Then, the Type~II multiple Laguerre polynomial of the second kind, for $\vec n=(n_1,\dots,n_r)\in\NN^r$, denoted by $L^{(\alpha,\vec c)}_{\vec n}$ where $\vec c =(c_1,\dots,c_r)\in(0,+\infty)^r$, is monic and satisfies
\begin{equation}\label{eq:mop-Laguerre-type-ii-1-var}
\langle L^{(\alpha,\vec c)}_{\vec n}, t^k\rangle_l = \int_0^{\infty} L^{(\alpha,\vec c)}_{\vec n}(t) \, t^k \, t^\alpha e^{-c_l t}\, \dr t =0\qquad \text{ if }0\leq k\leq n_l-1, \quad 1\leq l\leq r.
\end{equation}
On the other hand, the Type~I multiple Laguerre polynomials of the second kind are denoted by $B^{(\alpha,\vec c)}_{\vec n,1},\dots,B^{(\alpha,\vec c)}_{\vec n,r}$ with $\deg B^{(\alpha,\vec c)}_{\vec n,l}\leq n_l-1$ and satisfy
\begin{equation}\label{eq:mop-Laguerre-type-i-1-var}
    \sum_{l=1}^r \langle B^{(\alpha,\vec c)}_{\vec n,l}, t^k \rangle_l = \sum_{l=1}^r \int_0^\infty B^{(\alpha,\vec c)}_{\vec n,l}(t) \, t^k \, t^\alpha e^{-c_l t}\, \dr t = \begin{cases}
        0 & \text{ if } 0\leq k\leq |\vec n|-2,\\
        1  & \text{ if } k= |\vec n|-1.
    \end{cases}
\end{equation}

Explicit expressions as well as several properties and integral representations of these polynomials, together with applications in random matrices (Wishart ensemble) can be found in \cite{ABVA03}, \cite{BK05}, \cite[Section 23.4.2]{Ism05}, or \cite[Section 3.6.2]{VA20}. Furthermore, the asymptotics of multiple Laguerre polynomials, as well as their zero distribution, have been studied in \cite{LW08} and \cite{NVA16}.

Once multiple orthogonal polynomials have been introduced, we return to the single-measure but multivariate setting. In the following sections, we introduce what will become a key tool: orthogonal polynomials on the unit sphere and on the unit ball.

\subsection{Spherical harmonics}

Let $\mathbf x =(x_1,\dots,x_d)\in\RR^d$ and consider the euclidean norm $\|\mathbf x\|=\sqrt{x_1^2+\cdots+x_d^2}$. Denote by $\mathcal P_{n}^d$ the space of real homogeneous polynomials of degree $n$, and by $\Pi_n^d$ the space of real $d$-variate polynomials of degree at most $n$. Recall that \cite[Section 1.1]{FX13}
\begin{align}\label{eq:dimensions}
        \dim \mathcal P_n^d &= \binom{n+d-1}{n},& \dim\Pi_n^d =\binom{n+d}{n}.
\end{align}
    
Now, the Laplace operator is a two order differential operator defined as
$$
\Delta = \dfrac{\partial^2}{\partial x_1^2}+\cdots+\dfrac{\partial^2}{\partial x_d^2}.
$$
With these notations, we introduce the following definitions
\begin{definition}
For $n\geq 0$, the space of \textbf{real harmonic polynomials}  of degree $n$ on $\RR^d$ is defined as
$$
\mathbb H_n^d := \{P\in \mathcal P_n^d : \Delta P =0\}. 
$$
The \textbf{spherical harmonics} are the restrictions of the real harmonic polynomials to the unit sphere $\mathbb S^{d-1}$.
\end{definition}

Indeed, spherical harmonics are the restrictions of the elements of $\mathbb H_n^d $ to the sphere, but, following the standards \cite{FX13}, we will call $\mathbb H_n^d$ the space of spherical harmonics too. The dimension of this space is \cite[Corollary 1.1.4]{FX13}
    \begin{equation}
        \label{eq_dim-H}
        \dim \mathbb H_n^d = \dim \mathcal P_n^d-\dim \mathcal P_{n-2}^d =\binom{n+d-1}{n}-\binom{n+d-3}{n-2},
    \end{equation}
    where it is convenient to consider $\dim \mathcal P_{n-2}^d=0$ if $n=0,1$. 
    
    Recall the Lebesgue measure on the sphere $\mathbb S^{d-1}$. Fixing the dimension $d$ and $\{Y_\nu^n: 1\leq\nu\leq \dim \mathbb H_n^d\}$, $\{Y_\eta^m: 1\leq\eta\leq \dim \mathbb H_m^d\}$ orthogonal bases of $\mathbb H_n^d$ and $\mathbb H_m^d$ respectively, then \cite[Theorem 1.1.2]{FX13}
    \begin{equation}
        \label{eq:orthogonality-harmonics}
        \int_{\mathbb S^{d-1}}Y_\nu^n(\mathbf x)Y_\eta^m(\mathbf x)\,\dr \mathbf x = 0\qquad \text{ if }n\neq m\text{ or }\nu\neq\eta.
    \end{equation}

    Now, we show explicit expressions for spherical harmonics.

    \subsubsection{Spherical harmonics of two variables}\label{subsubsec:sph-harm-2}

    Whenever the dimension is $2$, in this section and in the whole paper, we will eventually express the points in $\RR^2$ using their polar coordinates, so that $\mathbf x = (x,y)\equiv(\rho,\theta)\in \RR^2$ where
    
    \begin{equation}
    \label{eq:polar-coordinates}
    \left.\begin{array}{rcl}
         x & = & \rho \cos\theta  \\
         y & = & \rho \sin\theta 
    \end{array}\right\}, \qquad 0\leq \rho\leq 1, \quad 0\leq \theta<2 \pi.
    \end{equation}
    The Jacobian of this transformation is $\rho$. This information will be useful for future changes of variable, see Section~\ref{sec:OP-disk}.

    In this bivariate setting, $\dim \mathbb H_n^2 = 2$. An orthogonal basis of $\mathbb H_n^2$ is given, in polar coordinates, by \cite[Section 1.6.1]{FX13}
    \begin{align}\label{eq:harmonics-2-var}
        Y_1^n(\rho,\theta) &= \rho^n \cos n\theta, & Y_2^n(\rho,\theta)=\rho^n\sin n\theta.        
    \end{align}
    Indeed, if we restrict the spherical harmonics to the unit circle $\mathbb S^1$, they reduce to
    \begin{align}\label{eq:harmonics-2-var-circle}
        Y_1^n(1,\theta) &= \cos n\theta, & Y_2^n(1,\theta)=\sin n\theta.    
    \end{align}

    \subsubsection{Spherical harmonics of three variables}\label{subsubsec:sph-harm-3}

    In the $3$-variate case, we use spherical coordinates to express points on $\RR^3$, so that we use the identification $\mathbf{x}=(x,y,z)\equiv(\rho,\theta,\phi)\in \RR^3$ where
    \begin{equation}
    \label{eq:sph-coordinates}
    \left.\begin{array}{rcl}
         x & = & \rho \sin\theta\sin\phi  \\
         y & = & \rho \sin\theta\cos\phi  \\
         z & = & \rho \cos\theta
    \end{array}\right\}, \qquad 0\leq \rho\leq 1, \quad 0\leq \theta\leq\pi,\quad 0\leq\phi<2 \pi.
    \end{equation}
    In this case, the Jacobian of the transformation to spherical coordinates is $\rho^2\sin\theta$.
    
    The space of spherical harmonics of degree $n$ and three variables $\mathbb H_n^3$ has dimension $2n+1$. In spherical coordinates, an orthogonal basis of $\mathbb H^3_n$ is (see \cite[Equation~1.6.5]{FX13} and \cite[Theorem 4.1.4]{DX14}).
    \begin{equation}
    \label{eq:sph-harmonics}
        \begin{aligned}
        Y_{k,1}^n(\rho,\theta,\phi)=\rho^n \sin^k\theta \, C_{n-k}^{k+1/2}(\cos(\theta))\,\cos (k\phi),\quad 0\leq k \leq n \\
        Y_{k,2}^n(\rho,\theta,\phi)=\rho^n \sin^k\theta \, C_{n-k}^{k+1/2}(\cos(\theta))\,\sin (k\phi),\quad 1\leq k \leq n, \\
        \end{aligned}
    \end{equation}
    where $\{C_n^{(\alpha)}(t)\}_{n\geq 0}$ are the Gegenbauer polynomials \cite[Section 4.7]{Sze59}, which are orthogonal on $[-1,1]$ with respect to the weight $(1-t^2)^{\alpha-1/2}$.
    
    Eventually, we will denote these polynomials as $Y^n_\nu$, where
    \begin{equation}
        \label{eq:sph-harmonics-nu}
        \begin{aligned}
        Y_{\nu}^n(\rho,\theta,\phi) &=Y_{\nu,1}^n(\rho,\theta,\phi)=\rho^n \sin^\nu(\theta) \, C_{n-\nu}^{\nu+1/2}(\cos(\theta))\,\cos(\nu\phi), & 0\leq \nu \leq n \\
        Y_{\nu}^n(\rho,\theta,\phi) &=Y_{\nu-n,2}^n(\rho,\theta,\phi)=\rho^n \sin^{\nu-n}(\theta) \, C_{2n-\nu}^{\nu-n+1/2}(\cos(\theta))\,\sin((\nu-n)\phi), & n+1\leq \nu \leq 2n, \\
        \end{aligned}
    \end{equation}
    so that $\{Y_\nu^n:\nu=0,\dots,2n\}$ is an orthogonal basis of $\mathbb H_n^3$. In addition, the restriction of the spherical harmonics to the unit sphere $\mathbb S^2$ reduces to $ Y_{k,1}^n(1,\theta,\phi)$ and $ Y_{k,2}^n(1,\theta,\phi)$.

    \subsubsection{Spherical harmonics of $d$ variables}\label{subsubsec:sph-harm-d}

    Now, in the general case, an arbitrary number $d\geq 2$ of variables is considered. For an arbitrary tuple $\mathbf{x}=(x_1,\dots,x_d)\in\RR^d$, we will consider its spherical polar coordinates, which generalize \eqref{eq:polar-coordinates} and \eqref{eq:sph-coordinates} identifying $\mathbf x =(x_1,\dots,x_d)\equiv (\rho,\theta_1,\dots,\theta_{d-1})$, where
    \begin{equation}
    \label{eq:sph-coordinates-d}
    \left.\begin{array}{rcl}
         x_1 & = & \rho \cos\theta_{d-1},  \\
         x_2 & = & \rho \sin\theta_{d-1} \cos\theta_{d-2},  \\
          & \vdots &  \\
         x_{d-1} & = &\rho \sin\theta_{d-1}\cdots \sin\theta_2 \cos \theta_1  \\
         x_d &=& \rho \sin\theta_{d-1}\cdots \sin\theta_2 \sin \theta_1
    \end{array}\right\},
    \quad \begin{array}{l}
         0\leq \rho\leq 1, \quad 0\leq\theta_1<2 \pi,  \\
         0\leq \theta_k\leq\pi,\ \ k=2,\dots,d-1 . 
    \end{array} 
    \end{equation}
    The Jacobian of this transformation is 
    \begin{equation}
        \label{eq:jacobian-sph-pol-coordinates}
        \prod_{k=1}^{d-2}\rho^{d-1}(\sin\theta_{d-k})^{d-k-1}.
    \end{equation}

    In order to define an orthogonal basis of $\mathbb H_n^d$, the space of spherical harmonics in $d$ variables of degree $n$, we will work with multi-indices $\mathbf m = (m_1,\dots,m_d)\in\NN^d_0$ such that $|\mathbf m|=n$ and $m_d\in\{0,1\}$. The number of possible tuples satisfying these conditions is exactly $\dim \mathbb H_n^d$, given in \eqref{eq_dim-H}, see \cite[Section 4.1]{DX14}. In spherical polar coordinates, an orthogonal basis of $\mathbb H_n^d$ is given by $\{Y_{\mathbf m}^n(\mathbf x):|\mathbf m|=n,\, m_d\in\{0,1\}\}$, where
    \begin{equation}
    \label{eq:sph-harmonics-d}
        Y_{\mathbf m}^n(\rho,\theta_1,\dots,\theta_{d-1})=\rho^n\, g(\theta_1)\, \prod_{k=1}^{d-2}\left[ (\sin\theta_{d-k})^{\beta_k}\, C_{m_k}^{\lambda_k}(\cos \theta_{d-k}) \right], 
    \end{equation}
    with $g(\theta_1)=\cos((m_{d-1}+m_d)\theta_1)$ if $m_d=0$ or $g(\theta_1)=\sin((m_{d-1}+m_d)\theta_1)$ if $m_d=1$, $\beta_k=\sum_{i=k+1}^d m_i$, and $\lambda_k = \beta_k + (d-k-1)/2$. Again, the restriction of these polynomials to the unit sphere $\mathbb S^{d-1}$ is $Y_{\mathbf m}^n(1,\theta_1,\dots,\theta_{d-1})$.

    Despite the generality of \eqref{eq:sph-harmonics-d} for any dimension $d\geq 2$, in the 2 or 3-variable setting it is common to use the expressions introduced in Sections~\ref{subsubsec:sph-harm-2} and \ref{subsubsec:sph-harm-3} for the sake of simplicity; see next two sections.

    \subsection{Orthogonal polynomials on the Unit Ball}

    One of the most studied domains in the theory of multivariate orthogonal polynomials is the unit ball of $\RR^d$, defined as
    \begin{equation}
        \label{eq:d-Ball}B^d = \{\mathbf x\in\RR^d:\|\mathbf x\|\leq 1\}.
    \end{equation}
    Choose $\mu>-1$ and define the following weight function on $B^d$:
    \begin{equation}
        \label{eq:weight-ball-d}
        W_\mu(\mathbf x)=w_\mu\,(1-\|\mathbf x\|^2)^{\mu}=w_\mu\,(1-x_1^2-\cdots-x_d^2)^\mu, \qquad \mathbf x = (x_1,\dots,x_d)\in B^d,
    \end{equation}
    where the normalization constant $w_\mu$ is 
    $$
    w_\mu= \left(\int_D W_\mu(\mathbf x)\, \dr\mathbf x\right)^{-1} = \frac{\Gamma(\mu+1+\frac{d}{2})}{\pi^{d/2}\Gamma(\mu+1)}.
    $$
    This weight leads to the inner product
    \begin{equation}
    \label{eq:inner-prod-d}
        \innerprod{f}{g}_\mu:=\int_{B^d}f(\mathbf x)g(\mathbf x)W_\mu(\mathbf x)\,\dr \mathbf x,
    \end{equation}
    for $f,g\in L^2(B^d;W_\mu)=\left\{ u(\mathbf x): \int_{B^d} (u(\mathbf x))^2W_\mu(\mathbf x)\,\dr \mathbf x<\infty\right\}$ \cite[Section 5.2]{DX14}.
  
    Fixed this weight and this inner product, we denote 
    $$
    \mathcal V_n^d=\mathcal V_n^d(W_\mu)=\{P\in\Pi_n^d:\innerprod{P}{Q}_{\mu}=0,\; \forall Q\in \Pi_{n-1}^d\}, 
    $$
    the space of polynomials of $d$ variables and degree $n$ which are orthogonal to lower degree polynomials with respect to the weight $W_\mu$.

    There exist several bases of $\mathcal V_n^d$, expressed in terms of Gegenbauer polynomials or Appell polynomials (see \cite[Sections 5.2.1, 5.2.2]{DX14}), but we will focus on the one given explicitly in terms of spherical harmonics. 
    
    More precisely, consider $n\geq 0$, $0\leq j\leq n/2$ and $\{Y_\nu^{n-2j}:1\leq \nu\leq \dim \mathbb H_{n-2j}^d\}$ an orthogonal basis of $\mathbb H_{n-2j}^d$. Then, the polynomials
    \begin{equation}
        \label{eq:polynomials-in-the-ball-d}
        P_{j,\nu}^n (\mathbf x)=P_j^{(\nu,n-2j+(d-2)/2)}(2\|\mathbf x\|^2-1) \, Y_{\nu}^{n-2j}(\mathbf x),
    \end{equation}
    where $P^{(\alpha,\beta)}_j$ denotes the Jacobi polynomial with parameters $\alpha,\beta>-1$ and degree $j$, form an orthogonal basis of $\mathcal V_n^d$, see \cite[Proposition 5.2.1]{DX14}. In particular, if we use the basis $\{Y_{\mathbf m}^n\}$ given in \eqref{eq:sph-harmonics-d}, for $\mathbf m =(m_1,\dots,m_{d})$ with $|\mathbf m|=n-2j$ and $m_d\in\{0,1\}$, we get the polynomials
    \begin{equation}
        \label{eq:polynomials-in-the-ball-d-2}
        P_{j,\mathbf m}^n (\mathbf x)=P_j^{(\nu,n-2j+(d-2)/2)}(2\|\mathbf x\|^2-1) \, Y_{\mathbf m}^{n-2j}(\mathbf x).
    \end{equation}
    In spherical polar coordinates:
    \begin{equation}
        \label{eq:polynomials-in-the-ball-d-2-sph}
        P_{j,\mathbf m}^n (\rho,\theta_1,\dots,\theta_{d-1})=P_j^{(\nu,n-2j+(d-2)/2)}(2\rho^2-1) \,\rho^{n-2j}\, g(\theta_1)\, \prod_{k=1}^{d-2}\left[ (\sin\theta_{d-k})^{\beta_k}\, C_{m_k}^{\lambda_k}(\cos \theta_{d-k}) \right],
    \end{equation}
    where $g(\theta_1)=\cos((m_{d-1}+m_d)\theta_1)$ if $m_d=0$ or $g(\theta_1)=\sin((m_{d-1}+m_d)\theta_1)$ if $m_d=1$, $\beta_k=\sum_{i=k+1}^d m_i$, and $\lambda_k = \beta_k + (d-k-1)/2$. \\

    The main purpose of this paper is to extend this constructions to the multiple orthogonal setting. This means, construct multivariate multiple orthogonal polynomials on the unit ball, leveraging the explicit expression \eqref{eq:polynomials-in-the-ball-d}, but substituting the standard Jacobi polynomials by one of their multiple orthogonal analogs: the Jacobi-Pi\~neiro polynomials already introduced in Section~\ref{subsec:JP}. 

    In the following section, we will introduce the simplest and most common case where $d=2$, that is, orthogonal polynomials on the unit disk.

    \subsubsection{Orthogonal polynomials on the Unit Disk}
    \label{sec:OP-disk}

    Fix $d=2$, so that $\mathbf x = (x,y)\equiv (\rho,\theta)$ in polar coordinates \eqref{eq:polar-coordinates}. We denote by 
    \begin{equation}
        \label{eq:disk}
        D =B^2= \{(x,y)\in\RR^2\; : x^2+y^2\leq 1 \},
    \end{equation}
    the unit disk of $\RR^2$  \cite[Section 2.3]{DX14}. Regarding the previous notations in \eqref{eq:weight-ball-d}, \eqref{eq:inner-prod-d}, we consider the following weight function 
    \begin{equation}
        \label{eq:weight-disk}
        W_{\mu}(x,y)=w_\mu (1-x^2-y^2)^\mu,  \qquad \mu>-1,
    \end{equation}
    where $w_\mu=(\mu+1)/\pi$, and its associated integral inner product
    \begin{equation}\label{eq:inner-prod-Cartesian}
        \innerprod{f}{g}_\mu=\int_D f(x,y)g(x,y)W_\mu(x,y)\dxdy, 
    \end{equation}
    for $f,g\in L^2(D;W_\mu)$. 
           
    If we express points on the disk $D$ using their polar coordinates \eqref{eq:polar-coordinates}, and use this change of variables on the inner product $\innerprod{\cdot}{\cdot}_\mu$, then \eqref{eq:inner-prod-Cartesian} might be written as
    \begin{equation}
    \label{eq:inner-prod-polar}
        \innerprod{f}{g}_\mu:=w_\mu\,\int_0^1 \int_0^{2\pi} f(\rho,\theta) g(\rho,\theta)\, r(1-\rho^2)^\mu\dr\theta\,\dr \rho,
    \end{equation}
    where $f(\rho,\theta)\equiv f(r\cos\theta,r\sin\theta)$ and $ g(\rho,\theta)\equiv g(r\cos\theta,r\sin\theta)$ are the functions $f$ and $g$ expressed in polar coordinates.

    Within this setup, we define the following polynomials, which are particular cases of \eqref{eq:polynomials-in-the-ball-d} for $d=2$ using \eqref{eq:harmonics-2-var}. 
    \begin{definition}
        Let $n\geq 0$. For $(x,y)\equiv (\rho,\theta) \in D$, define 
        \begin{equation}
        \label{eq:OP-disk}
        \begin{aligned}
            P^n_{j,1}(\rho,\theta) = P_j^{(n-2j,\mu)}(1-2\rho^2)\rho^{n-2j}\cos((n-2j)\theta), \quad 0\leq j \leq n/2
            \\
            P^n_{j,2}(\rho,\theta) = P_j^{(n-2j,\mu)}(1-2\rho^2)\rho^{n-2j}\sin((n-2j)\theta), \quad 0\leq j < n/2
        \end{aligned}
        \end{equation}
    \end{definition}
    
    It is known \cite[Proposition 2.3.3]{DX14} that $P^n_{j,1}$, $P^n_{j,2}$ are polynomials of degree $n$ in Cartesian coordinates. Moreover, as mentioned before, they form an orthogonal basis of $\mathcal{V}_n^2(W_\mu)$, that is:
    \begin{equation}
        \label{eq:pols-ball}
        \begin{aligned}
        \innerprod{P^n_{j,1}}{P^m_{k,1}}_\mu &= h_{j,1}^n\,\delta_{n,m}\delta_{j,k}, & 0\leq j\leq n/2, \ 0\leq k\leq m/2,\\
        \innerprod{P^n_{j,2}}{P^m_{k,2}}_\mu &= h_{j,2}^n\,\delta_{n,m}\delta_{j,k}, & 0\leq j< n/2, \ 0\leq k< m/2,\\
        \innerprod{P^n_{j,1}}{P^m_{k,2}}_\mu&=0, &0\leq j\leq n/2, \ 0\leq k< m/2,
        \end{aligned}
    \end{equation}
    where the squared norm $h_{j,\nu}^n$is given by
    $$
    h_{j,\nu}^n = \langle P^n_{j,\nu}, P^n_{j,\nu} \rangle_\mu = \dfrac{(\mu+1)}{2^{n-2j+\mu+2}\,\pi}\|P_j^{(n-2j,\mu)}\|^2 \, \|Y_\nu^{n-2j}(1,\theta)\|^2, \qquad \nu=1,2,
    $$
    with:
    \begin{equation}
        \label{eq:norm-jacobi}
        \|P_n^{(\alpha,\beta)}\|^2 = \dfrac{2^{\alpha+\beta+1}}{2n+\alpha+\beta+1}\dfrac{\Gamma(n+\beta+1)}{n! (n+\alpha+1)_{\beta}}
    \end{equation}
    (see \cite[Section 22.2.1]{AS64}), and
    \begin{align}
        \label{eq:norm-harmonics-2-var}
        \|Y_1^{n}(1,\theta)\|^2&= \pi(1+\delta_{n,0}) &
        \|Y_2^{n}(1,\theta)\|^2&= \pi
    \end{align}

    Within this preliminaries, once we studied Jacobi-Pi\~neiro polynomials, spherical harmonics, as well as orthogonal polynomials on the Unit Ball, it is time to gather these definitions and introduce Multiple Orthogonal Polynomials on the Unit Ball.

	\section{Multiple Orthogonal Polynomials on the Unit Ball}
	\label{sec:mop-ball}

    Starting from the theoretical foundations of multiple orthogonality and the properties of orthogonal polynomials on the unit ball, this section aims to integrate the multiple orthogonal properties of Jacobi-Pi\~neiro polynomials with the multivariate orthogonality on the unit ball. Specifically, our construction employs Jacobi-Pi\~neiro polynomials to govern the radial component, while spherical harmonics represent the angular part.

    Firstly, we will describe the case with $2$ variables. By providing an inductive approach to both the definitions and the multiple orthogonality relations, we aim to facilitate a clearer understanding of the generalization to higher-dimensional settings.
    
    \subsection{The bivariate case}\label{subsec:mop-2-vars}

    Let $\mu_1,\dots,\mu_r>-1$ such that $\mu_i-\mu_j\not\in\mathbb Z$ if $i\neq j$ and denote $\vec \mu =(\mu_1,\dots,\mu_r)$, $W_l(x,y) := W_{\mu_l}(x,y)$, and the inner products $\langle\cdot,\cdot\rangle_l := \langle\cdot,\cdot\rangle_{\mu_l}$, $l=1,\dots,r$, where $W_\mu$ is defined in \eqref{eq:weight-disk} and $\langle\cdot,\cdot\rangle_{\mu}$ in \eqref{eq:inner-prod-Cartesian}. In an analogous way as above, we define several functions in polar coordinates, which are polynomials in Cartesian coordinates and satisfy multiple orthogonal properties.
    
    \subsubsection{Type~II MOP on the disk}\label{subsec:type-ii-disk}

    Within the described setting, we will introduce an extension of orthogonal polynomials on the disk satisfying orthogonality relations with respect to the considered weights at the same time, playing a similar role as Type~II MOPs.

    \begin{definition}
    \label{def:MOP-2-vars}
        Let $\vec n = (n_1,\dots,n_r),\, \vec j = (j_1,\dots,j_r)\in\NN_0^r$ be such that $0\leq |\vec j| \leq |\vec n|/2$. Consider the Jacobi-Pi\~neiro Type~II MOP $P^{(|\vec n|-2|\vec j|, \vec \mu)}_{\vec j}(t)$, we define
        \begin{equation}
        \label{eq:mop-disk}
            \begin{array}{c}
            \mathbf P_{\vec j,1}^{\vec n}(\rho,\theta) =P^{(|\vec n|-2|\vec j|, \vec \mu)}_{\vec j}(\rho^2)\, \rho^{|\vec n|-2|\vec j|}\,\cos((|\vec n|-2|\vec j|)\theta), \qquad 0\leq |\vec j| \leq |\vec n|/2,
            \\[2ex]
            \mathbf P_{\vec j,2}^{\vec n}(\rho,\theta) = P^{(|\vec n|-2|\vec j|, \vec \mu)}_{\vec j}(\rho^2)\, \rho^{|\vec n|-2|\vec j|}\,\sin((|\vec n|-2|\vec j|)\theta), \qquad 0\leq |\vec j|< |\vec n|/2.
            \end{array}
        \end{equation}
    \end{definition}
    
    Now, recall from Euler's formula $(x+i\,y)^m=\rho^m(\cos m\theta + i \sin m\theta)$ that, as shown in \cite[Proof of Proposition 2.3.3]{DX14}, the functions $\rho^m \cos m\theta$ and $\rho^m \sin m\theta$ are polynomials of degree $m$ in Cartesian coordinates $(x,y)$. As a consequence, both $\mathbf P_{\vec j,1}^{\vec n}$ and $\mathbf P_{\vec j,2}^{\vec n}$ are polynomials of degree $|\vec n|$ in $(x,y)$.

    \begin{remark}\label{rem:notation-2}
        Observe that, as mentioned above, $\mathbf P_{\vec j,\nu}^{\vec n}(\rho,\theta)=P^{(|\vec n|-2|\vec j|, \vec \mu)}_{\vec j}(\rho^2) \,Y_\nu^{|\vec n|-2|\vec j|}(\theta)$, which bears an analogy with \eqref{eq:polynomials-in-the-ball-d}, substituting Jacobi polynomials by Jacobi-Pi\~neiro MOP. Besides, notice that Jacobi polynomials in \eqref{eq:polynomials-in-the-ball-d} are evaluated in $1-2\rho^2$ as they are orthogonal with respect to a weight supported on $[-1,1]$, whereas Jacobi-Pi\~neiro polynomials in \eqref{eq:mop-disk} -- and below -- are evaluated in $\rho^2$, since the multiple weights in \eqref{eq:JP-wights-1-var} are supported on $[0,1]$. 
    \end{remark}

    Now we present the multiple orthogonal properties satisfied by the polynomials in \eqref{eq:mop-disk}.

    \begin{proposition}\label{prop:orthogonality-2-vars}
        Let $\vec n = (n_1,\dots,n_r)$ and $\vec j = (j_1,\dots,j_r)\in\NN_0^r$ be such that  $0\leq |\vec j| \leq |\vec n|/2$ and consider the polynomials $\mathbf P_{\vec j,\nu}^{\vec n}$, $\nu=1,2$, defined in \eqref{eq:mop-disk}. Then
        \begin{equation}
        \label{eq:multiple-orthogonality-disk}
            \langle\mathbf P_{\vec j,\nu}^{\vec n}(x,y),\, x^a y^b\rangle_l=\int_D \mathbf P_{\vec j,\nu}^{\vec n}(x,y)\, x^a y^b \, W_l(x,y)\dxdy =0 \qquad \text{ if }0\leq a+b< |\vec n|-2|\vec j| + 2j_l,
        \end{equation}
        for $1\leq l\leq r$ and $\nu =1,2$. Moreover, 
        $$
        \langle\mathbf P_{\vec j,\nu}^{\vec n}(x,y),\, x^a y^b\rangle_l=0 \qquad\text{ if }a+b - |\vec n|+ 2|\vec j|\text{ is odd}, \qquad 1\leq l\leq r,  \ \nu =1,2.
        $$
    \end{proposition}
    
    \begin{proof}
        From the definition of $\mathbf P_{\vec j,\nu}^{\vec n}$  for $\nu=1,2$ given in  \eqref{eq:mop-disk}, 
        and using \eqref{eq:inner-prod-polar} and \eqref{eq:polar-coordinates}, the integral in \eqref{eq:multiple-orthogonality-disk} expressed in polar coordinates reads
        $$
        I= w_{\mu_l} \int_0^{2\pi}\int_0^1 P^{(|\vec n|-2|\vec j|, \vec \mu)}_{\vec j}(\rho^2)\, \rho^{|\vec n|-2|\vec j|}\,\cos((|\vec n|-2|\vec j|)\theta) \, \rho^{a+b+1} \,\cos^a\theta\, \sin^b\theta\, (1-\rho^2)^{\mu_l}\drdth, 
        $$
        if $\nu=1$, and 
        $$
        I =  w_{\mu_l} \int_0^{2\pi}\int_0^1 P^{(|\vec n|-2|\vec j|, \vec \mu)}_{\vec j}(\rho^2)\, \rho^{|\vec n|-2|\vec j|}\,\sin((|\vec n|-2|\vec j|)\theta) \, \rho^{a+b+1} \,\cos^a\theta\, \sin^b\theta\, (1-\rho^2)^{\mu_l}\drdth
        $$
        if $\nu=2$.
        
        This integral might be split into the product of two univariate integrals: one depending on $\rho$ (radial part), and another one depending on $\theta$ (angular part). As a consequence, $I=w_{\mu_l}\, I_1\times I_2$ where
        \begin{align*}
            I_1 &:= \int_0^1 P^{(|\vec n|-2|\vec j|, \vec \mu)}_{\vec j}(\rho^2)\, \rho^{|\vec n|-2|\vec j| + a+b+1}\,(1-\rho^2)^{\mu_l}\, \dr \rho \\
            I_2 &:= \begin{cases}
                \ds\int_0^{2\pi}  \cos((|\vec n|-2|\vec j|)\theta)\,\cos^a\theta\,\sin^b\,\dr\theta & \text{ if } \nu=1,\\
                \ds\int_0^{2\pi}  \sin((|\vec n|-2|\vec j|)\theta)\,\cos^a\theta\,\sin^b\,\dr\theta & \text{ if } \nu=2.
            \end{cases}
        \end{align*}
        If $a+b<|\vec n|-2|\vec j|$ or $a+b=|\vec n|-2|\vec j|+\ell$ with $\ell$ an odd integer, by Lemma~\ref{lemma:trig-integrals} (in Appendix~\ref{sec:appendix}), $I_2$ vanishes in both cases and, in turn, $I$ vanishes as well, giving us the second part of the theorem.

        Now assume $a+b = |\vec n|-2|\vec j|+\ell $ with $\ell$ an even integer. We will focus on $I_1$. Using the change of variable $t=\rho^2$, $\dr t = 2\rho\,\dr \rho$, $I_1$ can be rewritten as
        $$
        I_1 = \frac 1 2 \int_0^1 P^{(|\vec n|-2|\vec j|, \vec \mu)}_{\vec j}(t)t^{|\vec n|-2|\vec j|+\ell/2}(1-t)^{\mu_l}\,\dr t.
        $$
        Using the multiple orthogonality of $P^{(|\vec n|-2|\vec j|, \vec \mu)}_{\vec j}$ (see \eqref{eq:multiple-orthogonality-JP-type-ii}), we get $I_1=0$ if $\ell/2< j_l$, that is, if $a+b<  |\vec n|-2|\vec j| + 2 j_l$, $1\leq l\leq r$, which completes the proof.
                
    \end{proof}

    Once the definition and the properties of the Type~II MOP on the disk are established, we present a clarifying example.

        \begin{example}
        For simplicity, in this and the following examples we will consider only $r=2$ measures. As an example, let $\vec n = (2,4)$, $\vec j = (0,1)$, so that $M=|\vec n|-2|\vec j| = 4$. Polynomials $\mathbf P^{(2,4)}_{(0,1),1}(x,y)$ and $\mathbf P^{(2,4)}_{(0,1),2}(x,y)$ have both degree $|\vec n|=6$ and are orthogonal to every polynomial of degree lower than $M+2j_1 = 4$ (resp. $M+2j_2 = 6$)  with respect to the first measure (resp. second). This means, for a bivariate polynomial $q(x,y)$:
        \begin{align*}
            \langle \mathbf P^{(2,4)}_{(0,1),\nu}, q\rangle_1 &= 0  \text{ if }\deg(q)<4, & 
            \langle \mathbf P^{(2,4)}_{(0,1),\nu}, q\rangle_2 &= 0  \text{ if }\deg(q)<6.             
        \end{align*}
        Moreover,
        $$
        \langle \mathbf P^{(2,4)}_{(0,1),\nu}, q\rangle_l = 0  \text{ if }\deg(q)=4+(2k+1), \quad k\in\mathbb N_0, \quad \nu=1,2, \quad 1\leq j\leq r.
        $$
        
        Since the Jacobi-Pi\~neiro polynomial $P^{(4,(\mu_1,\mu_2))}_{(0,1)}(t)=5-(\mu_2+6) t$, the expressions in polar coordinates of these polynomials are
        \begin{align*}
            \mathbf P^{(2,4)}_{(0,1),1}(\rho,\theta) =  \rho ^4  \left(5-(\mu_2+6) \rho ^2\right)\, \cos (4 \theta )\\
            \mathbf P^{(2,4)}_{(0,1),2}(\rho,\theta) =  \rho ^4  \left(5-(\mu_2+6) \rho ^2\right)\, \sin (4 \theta ).
        \end{align*}
        In Cartesian coordinates these expressions become
        \begin{align*}
            \mathbf P^{(2,4)}_{(0,1),1}(x,y)&=\left(-x^4+6 x^2 y^2-y^4\right) \left((\mu_2+6) \left(x^2+y^2\right)-5\right) \\
            \mathbf P^{(2,4)}_{(0,1),2}(x,y)&=4 x y (y-x) (x+y) \left((\mu_2+6) \left(x^2+y^2\right)-5\right).
        \end{align*}
    \end{example}

    So far, by substituting Jacobi polynomials by Type~II Jacobi-Pi\~neiro polynomials we introduced Type~II MOP on the disk. However, there is no reason to limit the extension to Type~II MOP, as shown in the next section.

    \subsubsection{Type~I MOP on the disk}\label{subsubsec:type-i-disk}

    Following the same approach used for Type~II MOPs, we construct an extension of Type~I MOPs preserving the structure of \eqref{eq:polynomials-in-the-ball-d} and based on Type~I Jacobi-Pi\~neiro polynomials.
    
    \begin{definition}\label{def:type-i-2-vars}
        Let $\vec n = (n_1,\dots,n_r)$ and $ \vec j = (j_1,\dots,j_r)\in\NN_0^r$ be such that $0\leq |\vec j| \leq |\vec n|/2$. Consider Type~I Jacobi-Pi\~neiro MOPs $\left(A^{(|\vec n|-2|\vec j|, \vec \mu)}_{\vec j,1}(t),\dots,A^{(|\vec n|-2|\vec j|, \vec \mu)}_{\vec j,r}(t)\right)$. We define the bivariate polynomials $\left(\mathbf A_{\vec j,\nu,(1)}^{\vec n}, \dots, \mathbf A_{\vec j,\nu,(r)}^{\vec n}\right)$ given by 
        \begin{equation}
            \label{eq:Type-I-MOP-2-vars}
            \begin{aligned}
                \mathbf A_{\vec j,1,(l)}^{\vec n}(\rho,\theta)= A^{(|\vec n|-2|\vec j|, \vec \mu)}_{\vec j,l}(\rho^2)\, \rho^{|\vec n|-2|\vec j|}\, \cos((|\vec n|-2|\vec j|)\theta), \qquad 0\leq |\vec j| \leq |\vec n|/2,\\
                \mathbf A_{\vec j,2,(l)}^{\vec n}(\rho,\theta)= A^{(|\vec n|-2|\vec j|, \vec \mu)}_{\vec j,l}(\rho^2)\, \rho^{|\vec n|-2|\vec j|}\, \sin((|\vec n|-2|\vec j|)\theta), \qquad 0\leq |\vec j| < |\vec n|/2,\\
            \end{aligned}
        \end{equation}
        for $1\leq l\leq r$.
    \end{definition}
    
    \begin{remark}\label{remark:degree-type-I}
    Taking into account \eqref{eq:Type-I-MOP-2-vars} and since the univariate Type~I MOP $A^{(|\vec n|-2|\vec j|, \vec \mu)}_{\vec j,l}(t)$ have degree at most $j_l-1$, the degree of the bivariate Type~I MOP $\mathbf A_{\vec j,\nu,(l)}^{\vec n}(\mathbf x)$ is at most $|\vec n|-2(|\vec j|-j_l+1)$.
    \end{remark}

    Observe that the structures of the polynomials $\mathbf A_{\vec j,\nu,(l)}^{\vec n}$ and $\mathbf P_{\vec j,\nu}^{\vec n}$ are analogous, with the Type~II Jacobi-Pi\~neiro polynomial replaced by Type~I MOPs. Consequently, the orthogonality of the angular part will remains unchanged, while the main differences in the orthogonality relations arise from the radial part.

    \begin{proposition}\label{prop:Type-I-MOP-orthogonality}
        Let $\vec n = (n_1,\dots,n_r)$ and $ \vec j = (j_1,\dots,j_r)\in\NN_0^r$ be such that $0\leq |\vec j| \leq |\vec n|/2$ and consider the Type~I polynomials $\left(\mathbf A_{\vec j,\nu,(1)}^{\vec n}, \dots, \mathbf A_{\vec j,\nu,(r)}^{\vec n}\right)$  defined in \eqref{eq:Type-I-MOP-2-vars}. Then
        \begin{equation}
            \sum_{l=1}^r\langle  \mathbf A_{\vec j,\nu,(l)}^{\vec n}(x,y) , x^a y^b\rangle_l=  \int_D \sum_{l=1}^r \mathbf A_{\vec j,\nu,(l)}^{\vec n}(x,y) \, x^a y^b \, W_l(x,y) \dxdy =0 \quad \text{ if }  0\leq a+b< |\vec n|-2 .
        \end{equation}
        Moreover,
        $$
        \sum_{l=1}^r\langle  \mathbf A_{\vec j,\nu,(l)}^{\vec n}(x,y) , x^a y^b\rangle_l=  0 \quad\text{ if } a+b-|\vec n|+2|\vec j|\text{ is odd}, \quad \nu=1,2.
        $$
    \end{proposition}
    
    \begin{proof}

    Using polar coordinates, \eqref{eq:Type-I-MOP-2-vars}, and the linearity of the integral operator, we rewrite the target integral as
    \begin{equation}
    \label{eq:integral-type-I-split}
        \begin{aligned}
        & \sum_{l=1}^r w_{\mu_l}\int_0^1\int_0^{2\pi} A^{(|\vec n|-2|\vec j|, \vec \mu)}_{\vec j,l}(\rho^2)\, \rho^{|\vec n|-2|\vec j|+a+b+1}\, Y_\nu(1, \theta)\,\cos^a\theta\,\sin^b\theta\, (1-\rho^2)^{\mu_l}\,\dr\theta\,\dr\rho\\
        &= \left(\int_{0}^{2\pi}\cos^a\theta\,\sin^b\theta \, Y_\nu(1, \theta)\,\dr\theta \right)\sum_{l=1}^r w_{\mu_l}\int_0^1 A^{(|\vec n|-2|\vec j|, \vec \mu)}_{\vec j,l}(\rho^2)\, \rho^{|\vec n|-2|\vec j|+a+b+1}\, (1-\rho^2)^{\mu_l}\,\dr\rho.
        \end{aligned}
    \end{equation}
    where $Y_\nu(1,\theta)=\cos((|\vec n|-2|\vec j|)\theta)$ if $\nu=1$, and $Y_\nu(1,\theta)=\sin((|\vec n|-2|\vec j|)\theta)$ if $\nu=2$. As described in the proof of Proposition~\ref{prop:orthogonality-2-vars}, we use Lemma~\ref{lemma:trig-integrals}, in Appendix~\ref{sec:appendix}, to deduce that the angular integral vanishes if $a+b<|\vec n|-2|\vec j|$ and also if $a+b=|\vec n|-2|\vec j|+\ell$ with $\ell$ an odd positive integer. 

    If $a+b=|\vec n|-2|\vec j|+\ell$ with $\ell$ even, let us focus on the radial sum of integrals, where we apply the change of variables $t = \rho^2$, $\dr t = 2\rho\,\dr \rho$ and obtain
    \begin{equation}
        \label{eq:integral-type-I}
        \dfrac 1 2 \sum_{l=1}^r w_{\mu_l}\int_0^1 A^{(|\vec n|-2|\vec j|, \vec \mu)}_{\vec j,l}(t)\, t^{|\vec n|-2|\vec j| + \ell/2}\, (1-t)^{\mu_l}\,\dr t.
    \end{equation}
    Now, using the Type~I multiple orthogonality of the Jacobi-Pi\~neiro MOP $A^{(|\vec n|-2|\vec j|, \vec \mu)}_{\vec j,l}(t)$ with respect to the weights $w_{\mu_l}t^{|\vec n|-2|\vec j|}\, (1-t^2)^{\mu_l}$, we get that this sum is zero if $\ell/2<|\vec j|-1$. Equivalently, the integral vanishes whenever $a+b<|\vec n|-2$.
    \end{proof}

    As a consequence of the standard normalization of Type~I MOP, we might estimate the value of the integral in the case $a+b=|\vec n|-2$, which is given in the following result, whose proof is straightforward.
    
    \begin{corollary} In the setting of Proposition~\ref{prop:Type-I-MOP-orthogonality}, assume $a+b=|\vec n|-2$, then
    $$
    \int_D \sum_{l=1}^r \mathbf A_{\vec j,\nu,(l)}^{\vec n}(x,y) \, x^a y^b \, W_l(x,y) \dxdy = \begin{cases}
        0 &  \text{if } |\vec n| \text{ is even}\\
        \frac 1 2 \int_{0}^{2\pi}\cos^a\theta\,\sin^b\theta \, Y_\nu(1, \theta)\,\dr\theta & \text{if }|\vec n|\text{ is odd }
    \end{cases},
    $$
    where the expressions for the second case are in  \eqref{eq:lemma-linear-combination-cos} and \eqref{eq:lemma-linear-combination-sin}.
    \end{corollary}
    \begin{proof}
        Owing to the multiple orthogonality properties of Type~I Jacobi-Pi\~neiro MOP, the value of \eqref{eq:integral-type-I} is $1$ if $\ell/2=|\vec j|-1$, this means, $a+b=|\vec n|-2$. In that case, the angular part in \eqref{eq:integral-type-I-split} gives us the value of the integral, and using Lemma~\ref{lemma:trig-integrals} the result holds.
    \end{proof}
    
    Next, we present an example where Type~I polynomials for certain $\vec n$ and $\vec j$ are explicitly computed.

    \begin{example}
        In order to get more readable expressions, in this example, we will set the values  $\mu_1=1$, $\mu_2 = 1/2$. We choose $\vec n = (4,5)$, $\vec j =(2,2)$. With these chosen values, both polynomials $\mathbf A_{(2,2),\nu,(1)}^{(4,5)}$ and $\mathbf A_{(2,2),\nu,(2)}^{(4,5)}(x,y)$ have degree $|\vec n|-2(|\vec j|-j_1+1)=|\vec n|-2(|\vec j|-j_2+1)=3$ and satisfy
        $$
        \langle \mathbf A_{(2,2),\nu,(1)}^{(4,5)}, q\rangle_1 + 
        \langle \mathbf A_{(2,2),\nu,(2)}^{(4,5)}, q\rangle_2 =0 \quad\text{ if } \deg(q)<|\vec n|-2 = 7 \text{ or } \deg(q)=1+(2k+1),\ k\in\mathbb N_0,
        $$
        for any bivariate polynomial $q(x,y)$. The Type~I Jacobi-Pi\~neiro polynomials are
        \begin{align*}
            A^{(|\vec n|-2|\vec j|,(\mu_1,\mu_2))}_{(2,2),1} &= 30030 (t-2),& 
            A^{(|\vec n|-2|\vec j|,(\mu_1,\mu_2))}_{(2,2),2} &=-\frac{315315}{64} (11 t-12),
        \end{align*}
        which, according to Definition~\ref{def:type-i-2-vars}, lead to the following bivariate Type~I MOPs expressed in polar coordinates
        \begin{align*}
            \mathbf A_{(2,2),1,(1)}^{(4,5)}(\rho,\theta)&=30030 \rho  \left(\rho ^2-2\right) \cos (\theta ),& \mathbf A_{(2,2),1,(2)}^{(4,5)}(\rho,\theta)&=- \frac{315315}{64}\rho  \left(11 \rho ^2-12\right) \cos (\theta ),\\
            \mathbf A_{(2,2),2,(1)}^{(4,5)}(\rho,\theta)&=30030 \rho  \left(\rho ^2-2\right) \sin (\theta ),& \mathbf A_{(2,2),2,(2)}^{(4,5)}(\rho,\theta)&=- \frac{315315}{64}\rho  \left(11 \rho ^2-12\right) \sin (\theta ).
        \end{align*}
        In Cartesian coordinates they become
        \begin{align*}
            \mathbf A_{(2,2),1,(1)}^{(4,5)}(x,y)&=30030 x \left(x^2+y^2-2\right),& \mathbf A_{(2,2),1,(2)}^{(4,5)}(x,y)&=-\frac{315315}{64}  x \left(11 x^2+11 y^2-12\right)\\
            \mathbf A_{(2,2),2,(1)}^{(4,5)}(x,y)&=30030 y \left(x^2+y^2-2\right),& \mathbf A_{(2,2),2,(2)}^{(4,5)}(x,y)&=-\frac{315315}{64} y \left(11 x^2+11 y^2-12\right).
        \end{align*}
        
    \end{example}
    
    In the next section, we will extend these definitions and results to the the general case where $d\geq 2$ variables are considered.

    \subsection{The $d$ variables case}

    We move on to the case concerning an arbitrary number $d$ of variables. Let $\mu_1,\dots,\mu_r>-1$ such that $\mu_i-\mu_j\not\in\mathbb Z$ if $i\neq j$ and denote $\vec \mu =(\mu_1,\dots,\mu_r)$. Consider the weights $W_l(\mathbf x):=W_{\mu_l}(\mathbf x)$ introduced in \eqref{eq:weight-ball-d} and their associated inner products $\langle\cdot,\cdot\rangle_l := \langle\cdot,\cdot\rangle_{\mu_l}$ defined in \eqref{eq:inner-prod-d}, $1\leq l\leq r$. In an analogous way as studied in the 2D disk, we will define $d$-variate functions, which are polynomials satisfying both Type~I and Type~II multiple orthogonal relations with respect to the weights $W_1,\dots,W_r$.

    \subsubsection{Type~II MOP on the multidimensional ball}

    We begin by introducing, following the same structure as in Section~\ref{subsec:mop-2-vars}, a construction that will be shown to be an extension of the Type~II MOP in the $d$-dimensional ball.

    \begin{definition}
        Let $\vec n = (n_1,\dots,n_r)$ and $ \vec j = (j_1,\dots,j_r)\in\NN_0^r$ be such that $0\leq |\vec j| \leq |\vec n|/2$. Let us define $M:=|\vec n|-2|\vec j|$ and consider the Jacobi-Pi\~neiro MOP $P_{\vec j}^{(M+(d-2)/2,\vec \mu)}(t)$. Now let $\mathbf m =(m_1,\dots,m_d)\in\NN^d_0$ be such that $|\mathbf m|=M$, $m_d\in\{0,1\}$ and consider the spherical harmonic $Y_{\mathbf m}^{M}(\mathbf x)$ introduced in \eqref{eq:sph-harmonics-d}. We define
        \begin{equation}
            \label{eq:MOP-d-variables}
            \mathbf{P}_{\vec j,\mathbf m}^{\vec n}(\mathbf x) = P_{\vec j}^{(M+(d-2)/2,\vec \mu)}(\rho^2)\,Y_{\mathbf m}^{M}(\rho,\theta_1,\dots,\theta_{d-1}), \qquad
        \end{equation}
        where $\mathbf x=(x_1,\dots,x_d)\equiv(\rho,\theta_1,\dots,\theta_{d-1})\in B^d$.
    \end{definition}

    Once again, this function, when expressed in Cartesian coordinates $(x_1,\dots,x_d)$, is a polynomial of $d$ variables of degree exactly $|\vec n|$. Using an inductive argument, we prove that the polynomial $\mathbf{P}_{\vec j,\mathbf m}^{\vec n}$ is a $d$-variate extension of its univariate counterparts.

    \begin{proposition}\label{prop:type-ii-orthogonality-d-vars}  Let $\vec n = (n_1,\dots,n_r)$ and $\vec j = (j_1,\dots,j_r)\in\NN_0^r$ be such that $0\leq |\vec j| \leq |\vec n|/2$. Let $\mathbf a=(a_1,\dots,a_d), \,\mathbf m =(m_1,\dots,m_d)\in\NN^d_0$ be such that $|\mathbf m|=M=|\vec n|-2|\vec j|$, $m_d\in\{0,1\}$. Consider the polynomial $\mathbf{P}_{\vec j,\mathbf m}^{\vec n}(\mathbf x)$ in \eqref{eq:MOP-d-variables} and denote an arbitrary $d$-variate monomial as $\mathbf x^{\mathbf a} = x_1^{a_1}\cdots x_d^{a_d}$. Then
    \begin{equation}
        \label{eq:multiple-orthogonality-d}
        \langle\mathbf{P}_{\vec j,\mathbf m}^{\vec n}(\mathbf x), \mathbf x^{\mathbf a}\rangle_l =\int_{B^d} \mathbf{P}_{\vec j,\mathbf m}^{\vec n}(\mathbf x)\, \mathbf x^{\mathbf a} \,W_l(\mathbf x)\, \dr\mathbf x = 0\quad \text{ if }0\leq|\mathbf a|< M +2j_l, 
    \end{equation}
    for $l=1,\dots,r$. Moreover, \eqref{eq:multiple-orthogonality-d} also holds if $a_{d-k}+\cdots+a_{d} < m_{d-k}+\cdots + m_{d}$, or $a_{d-k}+\cdots+a_{d} -( m_{d-k}+\cdots + m_{d})$ is an odd positive integer, $k=1,\cdots,d$.
    
    \end{proposition}
    \begin{proof}
        In this proof, we will follow an inductive procedure. First, we express the integral in \eqref{eq:multiple-orthogonality-d} in $d$-dimensional spherical polar coordinates \eqref{eq:sph-coordinates-d}. To this end,  the expression of $\mathbf{P}_{\vec j,\mathbf m}^{\vec n}(\mathbf x)$ is written  in spherical coordinates as in \eqref{eq:MOP-d-variables}, the weight $W_l$ becomes $(1-\rho^2)^{\mu_l}$, and the monomial $\mathbf x^{\mathbf a}$ becomes
        \begin{equation}
            \label{eq:monomial-d-vars}
            \mathbf x^{\mathbf a}= x_1^{a_1}\cdots x_d^{a_d} = \rho^{|\mathbf a|}\prod_{k=1}^{d-1}(\cos\theta_{d-k})^{a_{k}}(\sin\theta_{d-k})^{\alpha_k},
        \end{equation}
        where $\alpha_k=\sum_{i=k+1}^d a_i$ (observe that $\alpha_{k-1}=a_k+\alpha_k)$. Also, recall the Jacobian of the transformation to spherical polar coordinates given in \eqref{eq:jacobian-sph-pol-coordinates}.
        
        Therefore, \eqref{eq:multiple-orthogonality-d} becomes
        \begin{equation*}
            \begin{aligned}
                I=\int_{B^d} \mathbf{P}_{\vec j,\mathbf m}^{\vec n}(\mathbf x)\, \mathbf x^{\mathbf a} \,W_l(\mathbf x)\, \dr\mathbf x = & \int_0^1\int_0^{\pi}\cdots\int_0^\pi\int_{0}^{2\pi} P_{\vec j}^{(M+(d-2)/2,\vec \mu)}(\rho^2)\, \rho^{M}\\&\times g(\theta_1)\, \prod_{k=1}^{d-2}\left[ (\sin\theta_{d-k})^{\beta_k}\, C_{m_k}^{\lambda_k}(\cos \theta_{d-k}) \right]\, \rho^{|\mathbf a|}\prod_{k=1}^{d-1}(\cos\theta_{d-k})^{a_{k}}(\sin\theta_{d-k})^{\alpha_k} \\
               &  \times\rho^{d-1} \prod_{k=1}^{d-2}(\sin\theta_{d-k})^{d-k-1} \,(1-\rho^2)^{\mu_l}\,\dr\theta_{1}\dr\theta_{2}\,\cdots\dr\theta_{d-1}\dr \rho,  
            \end{aligned}
        \end{equation*}
        with $\beta_k=\sum_{i=k+1}^d m_i$ (observe that $\beta_{k-1}=m_k+\beta_k$), and $\lambda_k = \beta_k + (d-k-1)/2$. The above integral might be split as a product of $d$ univariate integrals so that $I=I_\rho\times I_{\theta_1}\times I_{\theta_2}\times\cdots\times I_{\theta_{d-1}}$ with
        \begin{equation}
            I_\rho = \int_0^1 P_{\vec j}^{(M+(d-2)/2,\vec \mu)}(\rho^2)\, \rho^{M+|\mathbf a|+d-1}(1-\rho^2)^{\mu_l}\,\dr \rho,
        \end{equation}
        \begin{equation}\label{eq:I-theta_1}
            I_{\theta_1} = \int_0^{2\pi}g(\theta_1)(\cos\theta_1)^{a_{d-1}}(\sin\theta_1)^{a_d}\,\dr\theta_1,
        \end{equation}
        \begin{equation}\label{eq:I-theta_k}
            I_{\theta_k}=\int_{0}^\pi (\sin\theta_k)^{\beta_{d-k}+\alpha_{d-k}+k-1}(\cos\theta_k)^{a_{d-k}}C_{m_{d-k}}^{\lambda_{d-k}}(\cos\theta_k)\,\dr \theta_{k}, \quad k=2,\dots,d-1.
        \end{equation}

        Recall that $g(\theta_1)=\cos((m_{d-1}+m_d)\theta_1)$ if $m_d=0$ and $g(\theta_1)=\sin((m_{d-1}+m_d)\theta_1)$ if $m_d=1$. Then, by Lemma~\ref{lemma:trig-integrals} (in Appendix~\ref{sec:appendix}), $I_{\theta_1}=0$ (and in turn $I=0$) if $\alpha_{d-2}=a_{d-1}+a_d<\beta_{d-2}=m_{d-1}+m_d$ or $\alpha_{d-2}=\beta_{d-2}+\ell_1$, with $\ell_1$ an odd integer. Now, assume $\alpha_{d-2}=\beta_{d-2}+\ell_1$ with $\ell_1=\alpha_{d-2}-\beta_{d-2}$ an even integer.

        Focus on $I_{\theta_2}$ and use the change of variables $t=\cos\theta_2$, $\dr t = -\sin(\theta_2)\,\dr\theta_2$, keeping in mind that $\sin(\theta_2)=(1-t^2)^{1/2}$:
        $$
        I_{\theta_2}=\int_{0}^\pi (\sin\theta_2)^{\beta_{d-2}+\alpha_{d-2}+1}(\cos\theta_2)^{a_{d-2}}C_{m_{d-2}}^{\lambda_{d-2}}(\cos\theta_2)\,\dr \theta_{2} = -\int_{-1}^1 C_{m_{d-2}}^{\lambda_{d-2}}(t)\,t^{a_{d-2}}\,(1-t^2)^{(\beta_{d-2}+\alpha_{d-2})/2}\,\dr t.
        $$
        
        Gegenbauer polynomials $\{C_n^\alpha(t)\}_{n\geq 0}$ are orthogonal with respect to the weight $(1-t^2)^{\alpha-1/2}$ on $[-1,1]$. Since $\lambda_{d-2}=\beta_{d-2}+1/2$, the weight associated to Gegenbauer polynomials $C_{m_{d-2}}^{\lambda_{d-2}}$ is $(1-t^2)^{\beta_{d-2}}$, so that, using the orthogonality of Gegenbauer polynomials,
        $$
        I_{\theta_2}= -\int_{-1}^1 C_{m_{d-2}}^{\beta_{d-2}+1/2}(t)\,t^{a_{d-2}}\,(1-t^2)^{(\alpha_{d-2}-\beta_{d-2})/2}(1-t^2)^{\beta_{d-2}}\,\dr t =0
        $$
        if $\deg(t^{a_{d-2}}\,(1-t^2)^{(\alpha_{d-2}-\beta_{d-2})/2})=a_{d-2}+\alpha_{d-2}-\beta_{d-2}=\alpha_{d-3}-\beta_{d-2}<m_{d-2}$, or equivalently if $\alpha_{d-3}<\beta_{d-3}$.

        Moreover, since Gegenbauer polynomials are even (resp. odd) functions if their degree is even (resp. odd) and the interval $[-1,1]$ is symmetric, they are also orthogonal to monomials $x^k$ such that $k$ and the degree of the Gegenbauer polynomial have different parity, \textit{i.e.},
        \begin{equation}
            \label{eq:orthogonality-Gegenbauer}
            \int_{-1}^{-1} C_n^{\alpha}(t) x^k (1-t^2)^{\alpha-1/2}\,\dr t = 0 \qquad\text{ if } k<n,\quad \text{ or } k\geq n\text{ and }k-n\text{ is odd}. 
        \end{equation}
        
        Hence, if $\alpha_{d-3}=\beta_{d-3}+\ell_2$ with $\ell_2$ a non-negative integer, we might write the polynomial 
        $$t^{a_{d-2}}\,(1-t^2)^{(\alpha_{d-2}-\beta_{d-2})/2}=\sum_{i=0}^{(\alpha_{d-2}-\beta_{d-2})/2}c_i t^{2i+\alpha_{d-2}},$$
        so that
        $$
        I_{\theta_2}= -\sum_{i=0}^{(\alpha_{d-2}-\beta_{d-2})/2}c_i \int_{-1}^1 C_{m_{d-2}}^{\beta_{d-2}+1/2}(t)\, t^{2i+\alpha_{d-2}}\,    (1-t^2)^{\beta_{d-2}}\,\dr t.
        $$
        Recalling \eqref{eq:orthogonality-Gegenbauer},the $i$-th integral of the previous sum vanishes if $2i+a_{d-2}<m_{d-2}$. If $2i+a_{d-2}\geq m_{d-2}$, the $i$-th integral vanishes if $2i+a_{d-2}-m_{d-2}$ is odd. Since $2i$ is always even, then the $i$-th integral is zero, and in turn $I_{\theta_2}$ and $I$, whenever $a_{d-2}-m_{d-2}$ is odd.

        Since we are assuming that $\alpha_{d-3}=\beta_{d-3}+\ell_2$, or equivalently $a_{d-2}+\alpha_{d-2}=m_{d-2}+\beta_{d-2}+\ell_2$, and also that $\alpha_{d-2}-\beta_{d-2}$ is even, so that $(a_{d-2}-m_{d-2})+(\alpha_{d-2}-\beta_{d-2})=\ell_2$. Hence, $a_{d-2}-m_{d-2}$ is odd if, and only if $\ell_2$ is odd. Summarizing, $I_{\theta_2}=0$, and in turn $I=0$, if $\alpha_{d-3}=\beta_{d-3}+\ell_2$ with $\ell_2$ being an odd integer.

        Using induction, we will assume, for $2\leq k\leq d-2$, that $I = 0$ if $\alpha_{d-k-1}<\beta_{d-k-1}$ or $\alpha_{d-k-1}=\beta_{d-k-1}+\ell_{k}$ with $\ell_k$ an odd integer. With this assumption, we will prove that $I=0$ if $\alpha_{d-k-2}<\beta_{d-k-2}$ or $\alpha_{d-k-2}<\beta_{d-k-2}+\ell_{k+1}$ with $\ell_{k+1}$ an odd integer.

        We will study what happens if $\alpha_{d-k-1}=\beta_{d-k-1}+\ell_{k}$ with $\ell_k=\alpha_{d-k-1}-\beta_{d-k-1}$ an even integer (otherwise we know $I=0$ owing to the induction hypothesis). Focus on $I_{\theta_{k+1}}$ and use the change of variables $t=\cos\theta_{k+1}$, $\dr t = -\sin\theta_{k+1}\dr \theta_{k+1}$:
        \begin{equation*}
            \begin{aligned}
                I_{\theta_{k+1}}&= \int_{0}^\pi (\sin\theta_{k+1})^{\beta_{d-k-1}+\alpha_{d-k-1}+k}(\cos\theta_{k+1})^{a_{d-k-1}}C_{m_{d-k-1}}^{\lambda_{d-k-1}}(\cos\theta_{k+1})\,\dr \theta_{k+1}\\
                &= -\int_{-1}^1 C_{m_{d-k-1}}^{\lambda_{d-k-1}}(t)\,t^{a_{d-k-1}}\,(1-t^2)^{(\beta_{d-k-1}+\alpha_{d-k-1}+k-1)/2}\,\dr t .
            \end{aligned}
        \end{equation*}
        Since $\lambda_{d-k-1}=\beta_{d-k-1}+ k/2$, these Gegenbauer polynomials are orthogonal with respect to the weight $(1-t^2)^{\beta_{d-k-1}+ (k-1)/2}$. Rewrite $I_{\theta_{k+1}}$ as
        $$
        I_{\theta_{k+1}} = -\int_{-1}^1 C_{m_{d-k-1}}^{\lambda_{d-k-1}}(t)\,t^{a_{d-k-1}}\,(1-t^2)^{(\alpha_{d-k-1}-\beta_{d-k-1})/2}\,(1-t^2)^{\beta_{d-k-1}+ (k-1)/2}\,\dr t.
        $$
        By the orthogonality of Gegenbauer polynomials, see \eqref{eq:orthogonality-Gegenbauer}, we know that this integral vanishes whenever $\deg(t^{a_{d-k-1}}\,(1-t^2)^{(\alpha_{d-k-1}-\beta_{d-k-1})/2})=\alpha_{d-k-2}-\beta_{d-k-1}<m_{d-k-1}$, this means $\alpha_{d-k-2}<\beta_{d-k-2}$. If $\alpha_{d-k-2}=\beta_{d-k-2}+\ell_{k+1}$ with $\ell_{k+1}$ a non-negative integer, recalling \eqref{eq:orthogonality-Gegenbauer} and using an argument analogous to the one in the case of $I_{\theta_2}$, we get that $I_{\theta_{k+1}}$, and in turn $I$, vanish if $\alpha_{d-k-2}=\beta_{d-k-2}+\ell_{k+1}$ with $\ell_{k+1}$ an odd positive integer, completing the induction and obtaining the second part of the theorem.
        
        As a consequence, we finally get that $I=0$ if $\alpha_0=|\mathbf a| <\beta_0=|\mathbf m|=M=|\vec n|-2|\vec j|$ or $|\mathbf a|=M+\ell_{d-1}$ with $\ell_{d-1}$ an odd positive integer. 
        
        Finally, assume $|\mathbf a|=M+\ell_{d-1}$ with $\ell_{d-1}=|\mathbf a|-M$ even and focus on $I_\rho$. By the change of variables $t=\rho^2$, $\dr t =2\rho \dr \rho$:
        $$
        I_\rho = \int_0^1 P_{\vec j}^{(M+(d-2)/2,\vec \mu)}(\rho^2)\, \rho^{M+|\mathbf a|+d-1}(1-\rho^2)^{\mu_l}\,\dr \rho \int_0^1 P_{\vec j}^{(M+(d-2)/2,\vec \mu)}(t)\, t^{(2M+\ell_{d-1}+d-2)/2}(1-\rho^2)^{\mu_l}\,\dr \rho.
        $$
        Now, Jacobi-Pi\~neiro polynomials are multiple orthogonal with respect to the weights $t^{M+(d-2)/2}(1-t^2)^{\mu_l}$, $l=1,\dots,r$. Rewriting the above integral as
        $$
        I_\rho =   \int_0^1 P_{\vec j}^{(M+(d-2)/2,\vec \mu)}(t)\, t^{\ell_{d-1}/2}\, t^{M+(d-2)/2}(1-\rho^2)^{\mu_l}\,\dr \rho,
        $$
        we obtain that $I_\rho$, and in turn $I$, are zero whenever $\ell_{d-1}/2<j_l$, this means, whenever $|\mathbf a|<M+2j_l =|\vec n|-2|\vec j|+j_l$, completing the proof.
        
    \end{proof}

    Next, we will show an example of Type~II MOP on the $3$-dimensional unit ball $B^3$.

    \begin{example} 
        Assume $d=3$, so that we will be working on the 3D ball $$B^3=\{(x,y,z)\in\RR^3:x^2+y^2+z^2\leq 1\}.$$
        Instead of \eqref{eq:MOP-d-variables}, in order to simplify the notations, we will use the $3$-variate spherical harmonics introduced in \eqref{eq:sph-harmonics-nu}, which are the most common definition for spherical harmonics on 3 variables. Then, the Type~II MOP on $B^3$ will be denoted as
        \begin{equation}
        \label{eq:mop-ball}
        \mathbf{P}_{\vec j,\nu}^{\vec n}(\rho,\theta,\phi) = P_{\vec j}^{(|\vec n|-2|\vec j|+1/2,\vec \mu)}(\rho^2)\, Y_{\nu}^{|\vec n|-2|\vec j|}(\rho,\theta,\phi),\qquad 0\leq \nu\leq 2(|\vec n|-2|\vec j|),
        \end{equation}
        where $Y_\nu^{|\vec n|-2|\vec j|}$ are defined in \eqref{eq:sph-harmonics-nu}.

        Let $\vec n = (3,4)$, $\vec j=(2,1)$, so that $M=|\vec n|-2|\vec j| = 1$. We might consider $\mathbf P^{(3,4)}_{(2,1),\nu}$, with $\nu=0, 1, 2$. These polynomials have degree $|\vec n|=7$ and arise from the multiplication of Jacobi-Pi\~neiro Type~II MOP $P^{(M+1/2,(\mu_1,\mu_2))}_{(2,1)}$ evaluated in $\rho^2$ by the spherical harmonics $Y^M_{0,1}$, $Y^M_{1,1}$ and $Y^M_{1,2}$  respectively, see \eqref{eq:sph-harmonics} and \eqref{eq:sph-harmonics-nu}. They are orthogonal to the polynomials of degree lower than $M+2j_1=5$ (resp. $M+2j_2 = 3$) with respect to the first (resp. second) measure, \emph{i.e.}:
        \begin{align*}
             \langle \mathbf P^{(3,4)}_{(2,1),\nu},q\rangle_1 &=0 \text{ if }\deg(q)<5, & 
            \langle \mathbf P^{(3,4)}_{(2,1),\nu},q\rangle_2 &=0 \text{ if }\deg(q)<3.  
        \end{align*} 
        Moreover, 
        $$
        \langle \mathbf P^{(3,4)}_{(2,1),\nu},q\rangle_l =0 \text{ if }\deg(q)=\nu'+(2k+1), \quad k\in\mathbb N_0, \quad \nu=0,1,2, \quad 1\leq l\leq r.
        $$
        
        Since
        $$
        P^{(3/2,(\mu_1,\mu_2))}_{(2,1)}(t)=\frac{1}{8} (t ((2 \mu_1+11) t (9 (2 \mu_1+4 \mu_2+31)-(2 \mu_1+13) (2 \mu_2+11) t)-63 (4 \mu_1+2 \mu_2+29))+315),
        $$
        the expressions of these polynomials in spherical coordinates are
        \begin{multline*}
            \mathbf P^{(3,4)}_{(2,1),0}(\rho,\theta,\phi)=-\frac{1}{8} \rho  \cos (\theta ) ((2 \mu_1+11) (2 \mu_1+13) (2 \mu_2+11) \rho ^6\\-9 (2 \mu_1+11) \rho ^4 (2 \mu_1+4 \mu_2+31)+63 \rho ^2 (4 \mu_1+2 \mu_2+29)-315),
        \end{multline*}
        \begin{multline*}
            \mathbf P^{(3,4)}_{(2,1),1}(\rho,\theta,\phi)=-\frac{1}{8} \rho  \sin (\theta ) \cos (\phi ) ((2 \mu_1+11) (2 \mu_1+13) (2 \mu_2+11) \rho ^6\\ -9 (2 \mu_1+11) \rho ^4 (2 \mu_1+4 \mu_2+31)+63 \rho ^2 (4 \mu_1+2 \mu_2+29)-315),
        \end{multline*}
        \begin{multline*}
            \mathbf P^{(3,4)}_{(2,1),2}(\rho,\theta,\phi)= -\frac{1}{8} \rho  \sin (\theta ) \sin (\phi ) ((2 \mu_1+11) (2 \mu_1+13) (2 \mu_2+11) \rho ^6\\-9 (2 \mu_1+11) \rho ^4 (2 \mu_1+4 \mu_2+31)+63 \rho ^2 (4 \mu_1+2 \mu_2+29)-315).
        \end{multline*}
        In Cartesian coordinates these expressions become
        \begin{equation*}
            \begin{aligned}
                \mathbf P^{(3,4)}_{(2,1),0}(x,y,z)=&-\frac{1}{8} z ((2 \mu_1+11) (2 \mu_1+13) (2 \mu_2+11) \left(x^2+y^2+z^2\right)^3-9 (2 \mu_1+11)\times \\
                &\times (2 \mu_1+4 \mu_2+31) \left(x^2+y^2+z^2\right)^2+63 (4 \mu_1+2 \mu_2+29) \left(x^2+y^2+z^2\right)-315),\\
                \mathbf P^{(3,4)}_{(2,1),1}(x,y,z)=&-\frac{1}{8} y ((2 \mu_1+11) (2 \mu_1+13) (2 \mu_2+11) \left(x^2+y^2+z^2\right)^3-9 (2 \mu_1+11) \times \\
                & \times(2 \mu_1+4 \mu_2+31) \left(x^2+y^2+z^2\right)^2+63 (4 \mu_1+2 \mu_2+29) \left(x^2+y^2+z^2\right)-315),\\
                \mathbf P^{(3,4)}_{(2,1),2}(x,y,z)=&-\frac{1}{8} x ((2 \mu_1+11) (2 \mu_1+13) (2 \mu_2+11) \left(x^2+y^2+z^2\right)^3-9 (2 \mu_1+11) \times \\
                & \times(2 \mu_1+4 \mu_2+31) \left(x^2+y^2+z^2\right)^2+63 (4 \mu_1+2 \mu_2+29) \left(x^2+y^2+z^2\right)-315).
            \end{aligned}
        \end{equation*}
                
    \end{example}

    \subsubsection{Type~I MOP on the multidimensional ball}
    
    To close Section~\ref{sec:mop-ball}, we present a construction extending Definition~\ref{def:type-i-2-vars} to the general $d$-variate setting.

    \begin{definition}
        Let $\vec n = (n_1,\dots,n_r)$ and $ \vec j = (j_1,\dots,j_r)\in\NN_0^r$ be such that $0\leq |\vec j| \leq |\vec n|/2$. Define $M=|\vec n|-2|\vec j|$ and consider the Jacobi-Pi\~neiro Type~I MOPs $\left(A_{\vec j,1}^{(M+(d-2)/2,\vec \mu)}(t),\dots,A_{\vec j,r}^{(|\vec n|-2|\vec j|+(d-2)/2,\vec \mu)}(t)\right)$. Now let $\mathbf m =(m_1,\dots,m_d)\in\NN^d_0$ be such that $|\mathbf m|=M$, $m_d\in\{0,1\}$ and consider the spherical harmonic $Y_{\mathbf m}^{M}(\mathbf x)$ introduced in \eqref{eq:sph-harmonics-d}. We define the polynomials
        \begin{equation}
            \label{eq:type-i-mop-d-vars}
            \mathbf A_{\vec j,\mathbf m,(l)}^{\vec n}(\rho,\theta_1,\dots,\theta_{d-1}) = A_{\vec j,l}^{(M+(d-2)/2,\vec \mu)}(\rho^2)\, Y_{\mathbf m}^{M}(\rho,\theta_1,\dots,\theta_{d-1}),\quad 0\leq l \leq r,
        \end{equation}
        where $\mathbf x =(x_1,\dots,x_d)\equiv(\rho,\theta_1,\dots,\theta_{d-1})\in B^d$.
    \end{definition}

    Again, this function is a polynomial of degree at most $|\vec n|-2(|\vec j|-j_l+1)$ in Cartesian coordinates $(x_1,\dots,x_d)$. The Type~I multiple orthogonality properties fulfilled by these polynomials are presented below.

    \begin{proposition}\label{prop:type-i-orthogonality-d-vars}
        Let $\vec n = (n_1,\dots,n_r)$ and $ \vec j = (j_1,\dots,j_r)\in\NN_0^r$ be such that $0\leq |\vec j| \leq |\vec n|/2$. Let $\mathbf a=(a_1,\dots,a_d), \,\mathbf m =(m_1,\dots,m_d)\in\NN^d_0$ be such that $|\mathbf m|=|\vec n|-2|\vec j|$, $m_d\in\{0,1\}$. Consider the polynomials $\mathbf A_{\vec j,\mathbf m,(1)}^{\vec n},\dots, \mathbf A_{\vec j,\mathbf m,(r)}^{\vec n}$ in \eqref{eq:type-i-mop-d-vars} and denote an arbitrary $d$-variate monomial as $\mathbf x^{\mathbf a} = x_1^{a_1}\cdots x_d^{a_d}$. Then,
        \begin{equation}
            \label{eq:type-i-orthogonality-d-vars}
            \sum_{l=1}^r\langle\mathbf A_{\vec j,\mathbf m,(l)}^{\vec n}(\mathbf x), \mathbf x^{\mathbf a}\rangle_l=\int_{B^d}\sum_{l=1}^r \mathbf A_{\vec j,\mathbf m,(l)}^{\vec n}(\mathbf x)\, \mathbf x^{\mathbf a} \, W_l(\mathbf x)\, \dr \mathbf x =0\qquad\text{ if } 0\leq |\mathbf a|<|\vec n|-2.
        \end{equation}
        Moreover, \eqref{eq:type-i-orthogonality-d-vars} also holds if $a_{d-k}+\cdots+a_{d} < m_{d-k}+\cdots + m_{d}$, or $a_{d-k}+\cdots+a_{d} -( m_{d-k}+\cdots + m_{d})$ is an odd positive integer, $k=1,\cdots,d$.
    \end{proposition}
    \begin{proof}
    Using spherical polar coordinates \eqref{eq:sph-coordinates-d} and combining the definition of the polynomials $\mathbf A_{\vec j,\mathbf m,(l)}^{\vec n}$ in \eqref{eq:type-i-mop-d-vars}, the expression of a monomial $\mathbf x^{\mathbf a}$ in \eqref{eq:monomial-d-vars} and the Jacobian of the transformation \eqref{eq:jacobian-sph-pol-coordinates}, the integral becomes
    \begin{equation*}
        \begin{aligned}
            I=\int_{B^d}\sum_{l=1}^r \mathbf A_{\vec j,\mathbf m,(l)}^{\vec n}(\mathbf x)\, \mathbf x^{\mathbf a} \, W_l(\mathbf x)\, \dr \mathbf x = & \int_0^1\int_0^{\pi}\cdots\int_0^\pi\int_{0}^{2\pi} \sum_{l=1}^r  A_{\vec j,l}^{(|\vec n|-2|\vec j|+(d-2)/2,\vec \mu)}(\rho^2)\, \rho^{|\vec n|-2|\vec j|}\\&\times Y_{\mathbf m}^{|\vec n|-2|\vec j|}(1,\theta_1,\dots,\theta_{d-1})\, \rho^{|\mathbf a|}\prod_{k=1}^{d-1}(\cos\theta_{d-k})^{a_{k}}(\sin\theta_{d-k})^{\alpha_k} \\
           &  \times\rho^{d-1} \prod_{k=1}^{d-2}[(\sin\theta_{d-k})^{d-k-1}] \,(1-\rho^2)^{\mu_l}\,\dr\theta_{1}\dr\theta_{2}\,\cdots\dr\theta_{d-1}\dr \rho.  
        \end{aligned}
    \end{equation*}
    This might be split in $d$ integrals $I=I_{\rho}\times I_{\theta_1}\times\cdots\times I_{\theta_{d-1}}$ such that the expressions of $I_{\theta_1},I_{\theta_2},\dots, I_{\theta_{d-1}}$ are the same as in \eqref{eq:I-theta_1} and \eqref{eq:I-theta_k}, while
    $$
    I_\rho = \int_0^1\sum_{l=1}^r A_{\vec j,l}^{(|\vec n|-2|\vec j|+(d-2)/2,\vec \mu)}(\rho^2)\, \rho^{|\mathbf{a}|+|\vec n|-2|\vec j|+d-1}(1-\rho^2)^{\mu_l}\,\dr \rho.
    $$
    Following the arguments in the proof of Proposition~\ref{prop:type-i-orthogonality-d-vars}, we obtain that $I=0$ if $a_{d-k}+\cdots+a_{d} < m_{d-k}+\cdots + m_{d}$ or $a_{d-k}+\cdots+a_{d} -( m_{d-k}+\cdots + m_{d})$ is odd, $k=1,\cdots,d$. In particular, if $k=0$, we deduce that $I=0$ if $|\mathbf a|<|\mathbf m|=|\vec n|-2|\vec j|$, and also if $|\mathbf a|=|\vec n|-2|\vec j|+\ell$ with $\ell$ being an odd natural number.

    Suppose that $|\mathbf a|=|\vec n|-2|\vec j|+\ell$ with $\ell$ even and focus on $I_\rho$. Then, with this assumption and using the change of variable $t=\rho^2$, $\dr t =2\rho\, \dr \rho$, we write $I_\rho$ as
    $$
    \dfrac 1 2  \int_0^1\sum_{l=1}^r A_{\vec j,l}^{(|\vec n|-2|\vec j|+(d-2)/2,\vec \mu)}(t)\, t^{|\vec n|-2|\vec j|+(d-2)/2+\ell/2}(1-t^2)^{\mu_l}\,\dr t.
    $$
    Using the Type~I multiple orthogonality of the Jacobi-Pi\~neiro polynomials $A_{\vec j,l}^{(|\vec n|-2|\vec j|+(d-2)/2,\vec \mu)}(t)$ with respect to the weights $t^{|\vec n|-2|\vec j|+(d-2)/2}(1-t^2)^{\mu_l}$, $I_\rho=0$ if $\ell/2<|\vec j|-1$, which leads to $|\mathbf a|<|\vec n|-2$, completing the proof. 
    \end{proof}

    So far, we have provided analogous definitions for both Type~I and Type~II MOPs on multivariate domains, following a clear methodology that extends the ideas in \cite{DX14} by substituting the Jacobi polynomials in \eqref{eq:polynomials-in-the-ball-d} by one of their multiple orthogonal extensions, namely the Jacobi-Pi\~neiro polynomials. Once these definitions and their properties have been established, one of the main results in multiple orthogonality, the Nearest Neighbour Recurrence Relation \eqref{eq:NNRR-1-var}, will be extended to the multivariate setting.
    
    \section{Nearest Neighbour Relation in two variables}\label{sec:NNR}

    As mentioned in the preliminaries on multiple orthogonal polynomials in Section~\ref{sec:pre}, univariate families of MOPs satisfy Nearest Neighbour Recurrence Relations \eqref{eq:NNRR-1-var}, \eqref{eq:NNRR-1-var-2}, which extend the well-known three-term recurrence relation of standard orthogonal polynomials. Once definitions and multiple orthogonal properties of multiple orthogonal polynomials in the unit ball have been established throughout  Section~\ref{sec:mop-ball}, the way these relations extend to this multivariate setting is studied in this section.

    Although the results presented here can be extended to an arbitrary number of variables, we restrict our analysis to the bivariate case, described in Section~\ref{subsec:mop-2-vars}.

    Before extending the NNRR, we need an essential preliminary result, which is presented in the next section.

    \subsection{Biorthogonality}

    Let us define the Type~I function as
    \begin{equation}
        \label{eq:type-I-function}
        \mathbf Q_{\vec j,\nu}^{\vec n}(\mathbf x)=\sum_{l=1}^r \mathbf A_{\vec j,\nu,(l)}^{\vec n}(\mathbf x)\, W_l(\mathbf x),
    \end{equation}
    which allows to write the Type~I multiple orthogonality from Proposition~\ref{prop:Type-I-MOP-orthogonality} as
    $$
    \langle\mathbf Q_{\vec j,\nu}^{\vec n}(x,y),x^a y^b\rangle = \int_D \mathbf Q_{\vec j,\nu}^{\vec n}(x,y) \, x^a y^b \dxdy =0 \quad \text{ if }  0\leq a+b< |\vec n|-2.
    $$
    Where the inner product $\langle\cdot,\cdot\rangle$ (without sub-index) is the integral inner product with respect to the Lebesgue measure on $D$, \textit{i.e.}, with the constant weight function $1$. Indeed, in this equation, the weights $W_l$ have been omitted as they are explicitly included in the Type~I function. 
    
    By simultaneously leveraging  the properties of univariate Type~I and Type~II multiple orthogonality, we establish a biorthogonality relation between Type~I and Type~II MOPs on the disk.

    \begin{proposition} \label{prop:biorthogonality}
    
    Let $\vec n = (n_1,\dots,n_r),\,\vec m=(m_1,\dots,m_r), \, \vec j = (j_1,\dots,j_r),\,\vec k =(k_1,\dots,k_r)\in\NN_0^r$ be such that $0\leq |\vec j| \leq |\vec n|/2$ and $0\leq|\vec k|\leq|\vec m|/2$. Consider the Type~II polynomial $\mathbf{P}_{\vec j,\nu}^{\vec n}$ defined in \eqref{eq:mop-disk} and the Type~I function $\mathbf Q_{\vec k,\eta}^{\vec m}$ introduced in \eqref{eq:type-I-function}. Then 
    \begin{equation}
        I:=\int_D \mathbf{P}_{\vec j,\nu}^{\vec n}(\mathbf x)\, \mathbf Q_{\vec k,\eta}^{\vec m}(\mathbf x)\, \dr \mathbf x = 0\qquad \text{ if } \nu\neq \eta \text{ or } |\vec n|-2|\vec j|\neq |\vec m|-2|\vec k|.
    \end{equation}
    Moreover, if $\nu=\eta$ and $|\vec n|-2|\vec j| = |\vec m|-2|\vec k|$, then
    \begin{equation}
        I= \begin{cases}
            0 & \text{ if }  \vec k\leq \vec j \quad \text{(componentwise)},\\
            0 & \text{ if }  |\vec j|\leq |\vec k|-2, \\
            \|Y_{\nu}^{|\vec n|-2|\vec j|}\|^2/2 & \text{ if }  |\vec j| = |\vec k| -1,
        \end{cases}
    \end{equation}
    where the squared norm of the spherical harmonic is given in \eqref{eq:norm-harmonics-2-var}.
    \end{proposition}
    \begin{proof}
    From \eqref{eq:type-I-function} and \eqref{eq:Type-I-MOP-2-vars}, it is possible to express the Type~I function as
    $$
    \mathbf Q_{\vec k,\eta}^{\vec m}(\rho,\theta) = \sum_{l=1}^r  A^{(|\vec m|-2|\vec k|, \vec \eta)}_{\vec k,l}(\rho^2)\, \rho^{|\vec m|-2|\vec k|}\, Y_\eta^{(|\vec m|-2|\vec k|)}(1,\theta)\,(1-\rho^2)^{\mu_l},
    $$
    where $Y_\eta^{(|\vec m|-2|\vec k|)}(1,\theta)=\cos((|\vec m|-2|\vec k|)\theta)$ if $\eta=1$, and $Y_\eta^{(|\vec m|-2|\vec k|)}(1,\theta)=\sin((|\vec m|-2|\vec k|)\theta)$ if $\eta=2$. Using \eqref{eq:mop-disk} and splitting the radial and the angular part, $I$ can be expressed as $I=I_1 \times I_2$ where
    \begin{equation}
        \begin{aligned}
            I_1 &= \int_0^1 P^{(|\vec n|-2|\vec j|, \vec \mu)}_{\vec j}(\rho^2)\, \rho^{|\vec n|-2|\vec j|+|\vec m|-2|\vec k|+1} \sum_{l=1}^r  A^{(|\vec m|-2|\vec k|, \vec \eta)}_{\vec k,l}(\rho^2)\,(1-\rho^2)^{\mu_l}\,\dr\rho,\\
            I_2 &= \int_0^{2\pi}  Y_\nu^{(|\vec n|-2|\vec j|)}(1,\theta)\, Y_\eta^{(|\vec m|-2|\vec k|)}(1,\theta)\, \dr\theta.
        \end{aligned}
    \end{equation}
    We know that $I_2=\|Y_{\nu}^{|\vec n|-2|\vec j|}\|^2\, \delta_{\nu,\eta}\delta_{|\vec n|-2|\vec j|,|\vec m|-2|\vec k|}$, which gives us the first part of the result. 

    For the second one, assume $\nu=\eta$ and $|\vec n|-2|\vec j| = |\vec m|-2|\vec k|$ and apply the change of variables $t=\rho^2$, $\dr t = 2\rho \,\dr \rho$ to $I_1$, obtaining
    $$
    I_1 = \frac 1 2\int_0^1 P^{(|\vec n|-2|\vec j|, \vec \mu)}_{\vec j}(t)\,  \sum_{l=1}^r  A^{(|\vec n|-2|\vec j|, \vec \mu)}_{\vec k,l}(t)\, \,t^{|\vec n|-2|\vec j|}(1-t)^{\mu_l}\,\dr t.
    $$
    Observe that the univariate Type~I function for Jacobi-Pi\~neiro multiple orthogonal polynomials stands in $I_1$, so that
    $$
    I_1 = \frac 1 2\int_0^1 P^{(|\vec n|-2|\vec j|, \vec \mu)}_{\vec j}(t)\,  Q_{\vec k}^{|\vec n|-2|\vec j|,\vec \mu}(t)\,\dr t,
    $$
    which, using the biorthogonality relation between Type~I and Type~II MOP \cite[Theorem 23.1.6]{Ism05}, leads to the result.
    
    \end{proof}

    This biorthogonality relation between Type~I and Type~II MOP on the disk is a fundamental tool in order to prove what comes next: the extension of NNRR.

    \subsection{Nearest Neighbour Relation}
    
    Let $\vec n\in\NN^r$ and consider a path of neighbour multi-indices $\{\vec m_k \,:\, k=0,\dots,|\vec n| \}$ such that $ |\vec m_k|=k$, $\vec m_0 = \vec 0$, $\vec m_{|\vec n|}=\vec n$, $k=0,\dots,|\vec n|$.  
    Then
    \begin{equation}
        \label{eq:basis-MOP-n}
        \left\{\mathbf P_{\vec m_i,1}^{\vec n}: 0\leq i\leq |\vec n|/2\right\}\cup\left\{\mathbf P_{\vec m_i,2}^{\vec n}: 0\leq i< |\vec n|/2\right\}
    \end{equation}
    is a set of $|\vec n|+1$ lineally independent polynomials of degree $|\vec n|$. Repeating this consideration for every multi-index $\vec m_0,\dots,\vec m_{|\vec n|-1}$, we get
    \begin{equation}
        \label{eq:basis-MOP}
        \bigcup_{k=0}^{|\vec n|}\left\{\mathbf P_{\vec m_i,1}^{\vec m_k}: 0\leq i\leq k/2\right\}\cup\left\{\mathbf P_{\vec m_i,2}^{\vec m_k}: 0\leq i< k/2\right\}
    \end{equation}
    is a basis of $\Pi_{|\vec n|}^2$, the space of bivariate polynomials of degree at most $|\vec n|$. In this way, it is possible to express any bivariate polynomial as a linear combination of bivariate multiple orthogonal polynomials on the disk after choosing an appropriate path.
    
    In order to simplify the notation, we collect these polynomials into vectors and denote 
    \begin{equation}
    \label{eq:polynomial-vector-MOP}
    \begin{aligned}
        \mathbb P_{k} &:= \left(\mathbf P_{\vec m_0,1}^{\vec m_k},\dots,\mathbf P_{\vec m_{\lfloor k/2\rfloor },1}^{\vec m_k},\mathbf P_{\vec m_0,2}^{\vec m_k},\dots,\mathbf P_{\vec m_{\lfloor k/2\rfloor},2}^{\vec m_k}\right)& \text{if }k\text{ is odd,}\\
        \mathbb P_{k} &:= \left(\mathbf P_{\vec m_0,1}^{\vec m_k},\dots,\mathbf P_{\vec m_{k/2},1}^{\vec m_k},\mathbf P_{\vec m_0,2}^{\vec m_k},\dots,\mathbf P_{\vec m_{k/2-1},2}^{\vec m_k}\right)& \text{if }k\text{ is even.}
    \end{aligned}
    \end{equation}
    
    Observe that whether $k$ is odd or even, $\mathbb P_k$ is a vector of $k+1$ linearly independent polynomials of degree $k$.\\

    Combining the basis \eqref{eq:basis-MOP} with the multiple orthogonal properties of these polynomials, we present the following result, which extends the Nearest Neighbour Relations for Type~II MOPs \cite[Theorem 23.1.7]{Ism05}.\\

    \begin{theorem}\label{th:NNRR-2-vars}
        Let $\vec n$ and $\vec j$ be two multi-indices such that $0\leq |\vec j|\leq|\vec n|/2$ and let $\mathbf P_{\vec j,\nu}^{\vec n}$, with $\nu\in\{1,2\}$, be associated Type~II MOPs on the disk \eqref{eq:mop-disk}. Consider $\{\vec m_k \,:\, k=0,\dots,|\vec n| \}$ a path of neighbour multi-indices such that $ |\vec m_k|=k$, $\vec m_0 = \vec 0$, $\vec m_{|\vec n|}=\vec n$, $k=0,\dots,|\vec n|$. Choose $\vec w := \vec m_{|\vec n|+1}=\vec n+\vec e_l$ for some $l\in\{1,\dots,r\}$. For each multi-index $\vec m_k$, consider the polynomial vectors $\mathbb P_k$ given in \eqref{eq:polynomial-vector-MOP}.
        
        Then, the following relations hold
        \begin{equation}
        \label{eq:NNR}
            \begin{aligned}
                x\, \mathbf P_{\vec j,\nu}^{\vec n}(\mathbf x) = \mathbf c^{(|\vec n|+1)}\mathbb P_{|\vec n|+1} + \mathbf c^{(|\vec n|)}\mathbb P_{|\vec n|} +  \sum_{k=|\vec n|-2|\vec j|-1}^{|\vec n|-1}\mathbf c^{(k)}\mathbb P_{k}, \\
                y\, \mathbf P_{\vec j,\nu}^{\vec n}(\mathbf x) = \tilde{\mathbf c}^{(|\vec n|+1)}\mathbb P_{|\vec n|+1} + \tilde{\mathbf c}^{(|\vec n|)}\mathbb P_{|\vec n|} +  \sum_{k=|\vec n|-2|\vec j|-1}^{|\vec n|-1}\tilde{\mathbf c}^{(k)}\mathbb P_{k}, 
            \end{aligned}
        \end{equation}
        assuming that the sum starts with $k=0$ if $|\vec n|-2|\vec j|=0$, and where $\mathbf c^{(k)},\,\tilde{\mathbf c}^{(k)}\in\mathbb R^{k+1}$ are given by:
        \begin{equation}
        \label{eq:NNRR-coefficients-x}
        \begin{aligned}
            \mathbf c^{(k)} &=\begin{cases}
                (c^{(k)}_{0,1},\dots,c^{(k)}_{\lfloor k/2 \rfloor ,1}, c^{(k)}_{0,1}, \dots, c^{(k)}_{\lfloor k/2 \rfloor,1}) & \text{ if }k \text{ is odd},\\
                (c^{(k)}_{0,1},\dots,c^{(k)}_{k/2,1}, c^{(k)}_{0,1}, \dots, c^{(k)}_{k/2-1,1}) & \text{ if }k \text{ is even}, 
            \end{cases} &
            c^{(k)}_{i,\nu} = \dfrac{2}{\|Y_{\nu}^{k-2i}\|^2} \langle x\, \mathbf P_{\vec j,\nu}^{\vec n}, \mathbf Q^{\vec m_{k+2}}_{\vec m_{i+1},\nu}\rangle
        \end{aligned}
        \end{equation}
        \begin{equation}
        \label{eq:NNRR-coefficients-y}
        \begin{aligned}
            \tilde{\mathbf c}^{(k)} &=\begin{cases}
                (\tilde c^{(k)}_{0,1},\dots,\tilde c^{(k)}_{\lfloor k/2 \rfloor ,1}, \tilde c^{(k)}_{0,1}, \dots, \tilde c^{(k)}_{\lfloor k/2 \rfloor,1}) & \text{ if }k \text{ is odd},\\
                (\tilde c^{(k)}_{0,1},\dots,\tilde c^{(k)}_{k/2,1}, \tilde c^{(k)}_{0,1}, \dots, \tilde c^{(k)}_{k/2-1,1}) & \text{ if }k \text{ is even}, 
            \end{cases} &
            \tilde c^{(k)}_{i,\nu} = \dfrac{2}{\|Y_{\nu}^{k-2i}\|^2} \langle y\, \mathbf P_{\vec j,\nu}^{\vec n}, \mathbf Q^{\vec m_{k+2}}_{\vec m_{i+1},\nu}\rangle.
        \end{aligned}
        \end{equation}
        
    \end{theorem}
    
    \begin{proof}
        In this proof, we will show the first relation, the one multiplying by $x$, as the second one is analogous.
    
        As mentioned above, each multi-index $\vec m_k$ in the path from $\vec m_0=\vec 0$ to $\vec m_{|\vec n|}$ provides us $k+1$ lineally independent polynomials of degree $k$ \eqref{eq:basis-MOP-n}, so that \eqref{eq:basis-MOP} forms a basis for the polynomials of degree $|\vec n|$.
        
        Using this argument with $\vec w=\vec n + \vec e_l$, since $|\vec w|=|\vec n|+1$ and $x \, \mathbf P_{\vec j,\nu}^{\vec n}(\mathbf x)$ is a polynomial of degree $|\vec n|+1$, we might express it as a linear combination of Type~II MOP on the disk associated to the path $\vec m_0=\vec 0,\dots,\vec m_{|\vec n|}=\vec n, \vec m_{|\vec n|+1}=\vec w$, this is
        $$
        x\, \mathbf P_{\vec j,\nu}^{\vec n}(\mathbf x) =  \sum_{k=0}^{|\vec n|+1}\mathbf c^{(k)} \mathbb P_k(\mathbf x) =  \sum_{k=0}^{|\vec n|+1}\sum_{i=0}^{k/2}\sum_{\nu=1}^2 c_{i,\nu}^{(k)} \mathbf P_{\vec m_i,\nu}^{\vec m_k}(\mathbf x),
        $$
        where we are denoting as $\mathbf c^{(k)}$ the concatenation of the vectors $(c_{i,1}^{(k)}:0\leq i\leq k/2 )$ and $(c_{i,2}^{(k)}:0\leq i< k/2 )$.

        We will proof that $c^{(k)}_{i,\nu}=0$ whenever $k<|\vec n|-2|\vec j|-1$. Using the integral inner product with respect to the Lebesgue measure and multiplying the last expression by $\mathbf Q^{\vec m_t}_{\vec m_s,\eta}$, we have
        $$
        \langle x\, \mathbf P_{\vec j,\nu}^{\vec n}, \mathbf Q^{\vec m_t}_{\vec m_s,\eta}\rangle =  \sum_{k=0}^{|\vec n|+1}\sum_{i=0}^{k/2}\sum_{\nu=1}^2 c_{i,\nu}^{(k)} \langle  \mathbf P_{\vec m_i,\nu}^{\vec m_k},\mathbf Q^{\vec m_t}_{\vec m_s,\eta} \rangle.
        $$
        Using the biorthogonality relation from Proposition~\ref{prop:biorthogonality}, we have that $\langle  \mathbf P_{\vec m_i,\nu}^{\vec m_k},\mathbf Q^{\vec m_t}_{\vec m_s,\eta} \rangle=0$ if either $\nu\neq \eta$ or $k-2i \neq t-2 s$. Assume $\nu = \eta$ and $k-2i = t-2 s$, then $\langle \mathbf P_{\vec m_{i},\nu}^{\vec m_k}, \mathbf Q^{\vec m_t}_{\vec m_s,\nu} \rangle=0$ if $s\leq i$ (as $\vec m_s\leq\vec m_i$ iff $s\leq i$) or $|\vec m_i|=i\leq|\vec m_s|-2=s-2$. Consequently, the product vanishes unless $|\vec m_i|=i=|\vec m_s|-1=s-1$. In that case, substituting in $k-2i = t-2 s$, we get $k=t-2$. As a consequence, 
        $$
       \langle x\, \mathbf P_{\vec j,\nu}^{\vec n}, \mathbf Q^{\vec m_t}_{\vec m_s,\eta}\rangle= c^{(t-2)}_{s-1,\eta}\|Y_{\eta}^{t-2s}\|^2/2.
        $$
        From this equality, denoting $t=k+2, s=i+1$ we obtain the expressions of the coefficients in \eqref{eq:NNRR-coefficients-x}.
        
        The coefficient $c^{(t-2)}_{s-1,\eta}$ is zero whenever the product $\langle x\, \mathbf P_{\vec j,\nu}^{\vec n}, \mathbf Q^{\vec m_t}_{\vec m_s,\eta}\rangle$ is zero. Now observe
        $$
        \langle x\, \mathbf P_{\vec j,\nu}^{\vec n}, \mathbf Q^{\vec m_t}_{\vec m_s,\eta}\rangle = \langle \mathbf P_{\vec j,\nu}^{\vec n}, x\, \mathbf Q^{\vec m_{t}}_{\vec m_{s},\eta}\rangle = \sum_{l=1}^r \langle \mathbf P_{\vec j,\nu}^{\vec n}, x\, \mathbf A^{\vec m_{t}}_{\vec m_{s},\eta,(l)}\rangle_l. 
        $$
        By Type~II multiple orthogonality, see Proposition~\ref{prop:orthogonality-2-vars}, each product $\langle \mathbf P_{\vec j,\nu}^{\vec n}, x\, \mathbf A^{\vec m_{t}}_{\vec m_{s},\eta,(l)}\rangle_l$ vanishes if 
        \begin{equation}
            \label{eq:NNRR-1}
        \deg( x\, \mathbf A^{\vec m_{t}}_{\vec m_{s},\eta,(l)})<|\vec n|-2(|\vec j|-j_l).
        \end{equation}
        Recalling Remark~\ref{remark:degree-type-I},  
        \begin{equation*}
            \begin{aligned}
                \deg(x\, \mathbf A^{\vec m_{t}}_{\vec m_{s},\eta,(l)})&\leq  1 + |\vec m_t|-2(|\vec m_s|-(\vec m_s)_l+1) \\ 
                &=1+t-2(s-(\vec m_{s})_l+1)\\&=t-1-2(s-(\vec m_{s})_l),
            \end{aligned}
        \end{equation*}
        where we used $|\vec m_k|=k$.
          
        Since $(s-(\vec m_{s})_l)=|\vec m_{s}|-(\vec m_{s})_l\geq 0$, the minimum reachable value for $(s-(\vec m_{s})_l)$ is $0$, arising whenever $s=0$, in whose case $\vec m_{s}=\vec 0$. With this information, the inequality 
        \begin{equation}
            \label{eq:NNRR-12}
            \deg( x\, \mathbf A^{\vec m_{t}}_{\vec m_{s},\eta,(l)})\leq t-1-2(s-(\vec m_{s})_l) \leq t-1
        \end{equation}
        holds for every $l$.
        
        On the other hand, we have the right hand side of \eqref{eq:NNRR-1}. Since $\vec j$ might be any multi-index satisfying $0\leq |\vec j|\leq|\vec n|/2$, all we can assert is that $0\leq j_l\leq|\vec n|/2$. In this way, 
        \begin{equation}
            \label{eq:NNRR-3}
            |\vec n|-2(|\vec j|-j_l) \geq |\vec n|-2|\vec j|.
        \end{equation}
        Gathering \eqref{eq:NNRR-12} and \eqref{eq:NNRR-3}, we get that if $t-1<|\vec n|-2|\vec j|$, then \eqref{eq:NNRR-1} holds for every $l$, and consequently $c^{(t-2)}_{s-1,\eta}=0$. Equivalently, $c^{(k)}_{i,\eta}=0$ if $k<|\vec n|-2|\vec j|-1$, completing the proof.

    \end{proof}

    Observe that Equations~\eqref{eq:NNR} become more useful as $|\vec n|-2|\vec j|$ increases, since more coefficients vanish. However, these relations depend strongly on the multi-index $\vec j$ and on its modulus. This is due to the fact that $\vec j$ is not fixed, as it may be any  multi-index satisfying $0\leq|\vec j|\leq|\vec n|/2$. 

    We now study a particular case of these relations by imposing a specific structure on the path. This fact allows us to control the components of the multi-indices, at the cost of reducing the  freedom in their choice. Then, consider a path $\{\vec m_k:k\geq 0\}$ of neighbour multi-indices such that $|\vec m_k|=k$, $\vec m_0 = \vec 0$, $\vec m_1 = (1,0,\dots,0)$, $\vec m_2=(1,1,\dots,0)$, \dots , $\vec m_r=(1,1,\dots,1)$, $\vec m_{r+1}=(2,1,\dots,1)$, \dots ,\linebreak $\vec m_{ra+i}=(a+1,\overset{(i)}{\dots},a+1,a,\dots,a)$, and so on; see Figure~\ref{fig:step-line} for a graphic representation. This path is known as the \emph{step-line} and is  widely used to standardize the indexing of multiple orthogonal polynomials, see \cite{BDFMAF23, MRW25}. 
    \begin{figure}
        \centering
        \includegraphics[width=0.4\linewidth]{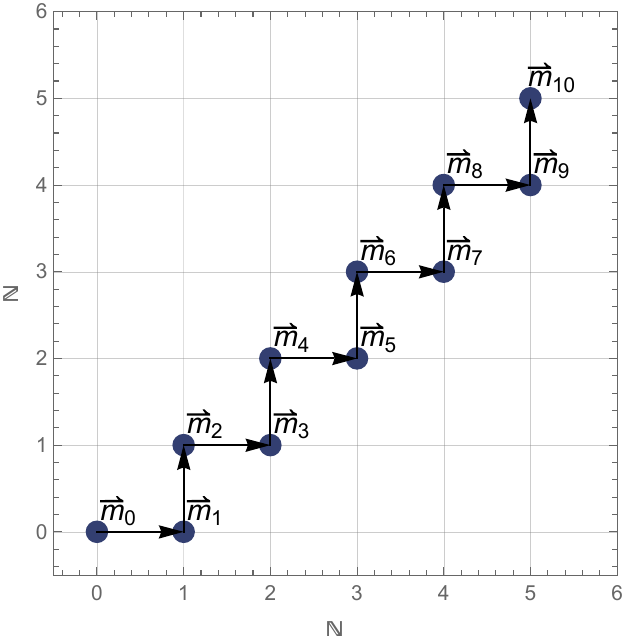}
        \caption{Step-line muti-indices for $r=2$}
        \label{fig:step-line}
    \end{figure}

    Let $\vec n=(n_1,\dots,n_r)\in\NN^r$ be a multi-index on the step-line, and assume that $|\vec n| = ar+i$, where $a = \lfloor|\vec n|/r\rfloor$ and $i = |\vec n|\,\mathrm{mod} \, r$. Then $n_j = a+1$ if $j\leq i$, and $n_j = a$ if $j> i$. Consequently, each component of $\vec n$ satisfies $\lfloor|\vec n|/r\rfloor\leq n_j \leq \lfloor|\vec n|/r\rfloor+1$. Using this characterization of the components of a multi-index on the step-line, we now present the following result.

    \begin{corollary}
        In the same setting as Theorem~\ref{th:NNRR-2-vars}, assume $\vec n$ is in the step-line, as well as the path $\vec m_0=\vec 0, \vec m_1,\dots,\vec m_{|\vec n|}=\vec n$ goes through the step-line. Then,
        \begin{equation}
        \label{eq:NNR-stepline}
            \begin{aligned}
                x\, \mathbf P_{\vec j,\nu}^{\vec n}(\mathbf x) = \mathbf c^{(|\vec n|+1)}\mathbb P_{|\vec n|+1} + \mathbf c^{(|\vec n|)}\mathbb P_{|\vec n|} +  \sum_{k=2\lfloor |\vec n|/(2r)\rfloor-1}^{|\vec n|-1}\mathbf c^{(k)}\mathbb P_{k}, \\
                y\, \mathbf P_{\vec j,\nu}^{\vec n}(\mathbf x) = \tilde{\mathbf c}^{(|\vec n|+1)}\mathbb P_{|\vec n|+1} + \tilde{\mathbf c}^{(|\vec n|)}\mathbb P_{|\vec n|} +  \sum_{k=2\lfloor |\vec n|/(2r)\rfloor-1}^{|\vec n|-1}\tilde{\mathbf c}^{(k)}\mathbb P_{k}, 
            \end{aligned}
        \end{equation}
        where the expressions of $\mathbf c^{(k)},\,\tilde{\mathbf c}^{(k)}$ are given in \eqref{eq:NNRR-coefficients-x}, \eqref{eq:NNRR-coefficients-y}.

    \end{corollary}
    \begin{proof}
        This proof is analogous to that of Theorem~\ref{th:NNRR-2-vars}, but makes use of the structure of the multi-indices to get another lower bound for $|\vec n|-2(|\vec j|-j_l)$.

        Since $\vec j$ is an element from the step-line, then $j_l\geq \lfloor|\vec j|/r\rfloor$, so that
        \begin{equation}
            \label{eq:NNRR-2}
            |\vec n|-2(|\vec j|-j_l)\geq |\vec n| -2(|\vec j|-\lfloor|\vec j|/r\rfloor),
        \end{equation}

        Now, we will seek an upper bound for $|\vec j|-\lfloor|\vec j|/r\rfloor$. For this purpose, the  function $f(t)=t-\lfloor t/r\rfloor$ is defined on $\mathbb N_0$. This function is non-decreasing, as
        $$
        f(0)=0, f(1)=1,\dots, f(r-1)=r-1,f(r)=r-1, f(r+1)=r,\dots
        $$
        and, in general if $t=ar+i$ with $a=\lfloor t/r\rfloor$ and $0\leq i\leq r-1$, then $f(t)=a(r-1)+i$.

        Then $f(|\vec j|)=|\vec j|-\lfloor|\vec j|/r\rfloor$ will reach its maximum at the maximum value of $|\vec j|$, which is $|\vec n|/2$. Then $|\vec j|-\lfloor |\vec j|/r\rfloor \leq |\vec n|/2- \lfloor |\vec n|/(2r)\rfloor$. Applying this inequality in \eqref{eq:NNRR-2}:
        $$
        |\vec n|-2(|\vec j|-j_l)\geq |\vec n| -2(|\vec j|-\lfloor|\vec j|/r\rfloor)\geq |\vec n| -2( |\vec n|/2- \lfloor |\vec n|/(2r)\rfloor)=2\lfloor |\vec n|/(2r)\rfloor.
        $$
        Combining these inequalities with \eqref{eq:NNRR-12}, we get that \eqref{eq:NNRR-1} holds for every $l$ whenever $t< 2\lfloor |\vec n|/(2r)\rfloor+1$. In turn, $c^{(t-2)}_{s-1,\eta}=0$ whenever that condition holds, and equivalently $c^{(k)}_{i,\nu} =0$ for $0\leq k < 2\lfloor |\vec n|/(2r)\rfloor-1$, completing the proof.
    \end{proof}    

    Observe that we have removed the dependence on $|\vec j|$ using the knowledge about its components. Let us show an example to illustrate the possible advantages of this case.

    \begin{example}
        Using $r=2$, let $\vec n = (4,4)$, $\vec j =(2,1)$ and consider the path $\vec m_0,\dots,\vec m_{|\vec n|}=\vec n$ in the step-line (see Figure~\ref{fig:step-line}), and choose either $\vec w=(5,4)$ or $\vec w=(4,5)$. Let $\mathbb P_0,\dots,\mathbb P_{|\vec n|+1}$ be the polynomial vectors defined in \eqref{eq:polynomial-vector-MOP}. Then, \eqref{eq:NNR} tells us that
        \begin{equation*}
            x\, \mathbf P_{(2,1),\nu}^{(4,4)}(\mathbf x) = \mathbf c^{(|\vec n|+1)}\mathbb P_{|\vec n|+1} + \mathbf c^{(|\vec n|)}\mathbb P_{|\vec n|} +  \sum_{k=1}^{|\vec n|-1}\mathbf c^{(k)}\mathbb P_{k},
        \end{equation*}
        since $|\vec n|-2|\vec j|-1 = 1$. However, according to \eqref{eq:NNR-stepline},
        \begin{equation*}
            x\, \mathbf P_{(2,1),\nu}^{(4,4)}(\mathbf x) = \mathbf c^{(|\vec n|+1)}\mathbb P_{|\vec n|+1} + \mathbf c^{(|\vec n|)}\mathbb P_{|\vec n|} +  \sum_{k=3}^{|\vec n|-1}\mathbf c^{(k)}\mathbb P_{k},
        \end{equation*}
        as $2\lfloor|\vec n|/(2r)\rfloor-1 =3$, giving us a `shorter tail'.
    \end{example}

    To close this paper, in the next section we briefly introduce an extension of this framework: multiple orthogonal polynomials for radial weights supported on $\mathbb R^2$.
    
    \section{Multiple Orthogonal Polynomials for radial weights}\label{sec:radial-weights}

    In this section, the notion of MOP on the disk, studied in Section~\ref{subsec:mop-2-vars}, is extended to general radial weights. A weight function $W(\mathbf x)$ is said to be \emph{radial} if it is of the form $W(\mathbf x)=\omega(\rho)=\omega(\sqrt{x_1^2+\cdots + x_d^2})$, where we use spherical polar coordinates $\mathbf x =(x_1,\dots,x_d)\equiv(\rho,\theta_1,\dots,\theta_{d-1})$, introduced in \eqref{eq:sph-coordinates-d} \cite[Sections 2.6.2 and 5.1.2]{DX14}. 
       
    Now consider $r$ radial weights $W_1(\mathbf x)=\omega_1(\rho),\dots,W_r(\mathbf x)=\omega_r(\rho)$. For $k\geq 0$, define the system of weights $\rho^{k+1}\omega_1(\rho), \dots, \rho^{k+1}\omega_r(\rho)$ supported on $[0,+\infty)$. Let $\vec n$ be a normal multi-index for this system, we denote by $p_{\vec n}^{(k)}$ the associated Type~II MOP and by $q_{\vec n,l}^{(k)}$, $1\leq l\leq r$ the associated Type~I MOPs. This means that
    \begin{equation}
        \label{eq:type-ii-orthogonality-radial-aux}
        \int_0^{\infty} p_{\vec n}^{(k)}(\rho)\, \rho^t\, \rho^{k+1}\omega_l(\rho)\,\dr \rho = 0\qquad\text{ if }0\leq t<n_l, \quad 1\leq l\leq r,
    \end{equation}
    and
    \begin{equation}
        \label{eq:type-i-orthogonality-radial-aux}
        \sum_{l=1}^r \int_0^{\infty} q_{\vec n,l}^{(k)}(\rho)\,\rho^t\,\rho^{k+1}\,\omega_l(\rho)\,\dr \rho = \delta_{t,|\vec n|-1}, \qquad 0\leq t \leq |\vec n|-1.
    \end{equation}
    Using these auxiliary polynomials, we now present the Type~II and Type~I MOPs for the radial weights $W_1,\dots, W_r$.

    \subsection{Type~II MOP for radial weights}

    Following the approach in Section~\ref{subsec:type-ii-disk}, we define polynomials associated with radial weights, analogous to the disk case.

    \begin{definition}  Let $\vec n = (n_1,\dots,n_r)$ and $\vec j = (j_1,\dots,j_r)\in\NN_0^r$ be such that $0\leq |\vec j| \leq |\vec n|/2$ and consider the polynomial $p_{2\vec j}^{(2|\vec n|-4|\vec j|)}$ satisfying \eqref{eq:type-ii-orthogonality-radial-aux}, and  $Y_\nu^{|\vec n|-2|\vec j|}$, $\nu=1,\dots,\dim(\mathbb H_{|\vec n|-2|\vec j|}^d)$ a basis of $\mathbb H_{|\vec n|-2|\vec j|}^d$. Then, we define
    \begin{equation}
        \label{eq:type-ii-radial}
        \mathbf P^{\vec n}_{\vec j,\nu}(\mathbf x) = p_{2\vec j}^{(2|\vec n|-4|\vec j|)}(\rho)\, Y_\nu^{|\vec n|-2|\vec j|}(\mathbf x).
    \end{equation}
    \end{definition}
    
    This function is a polynomial of degree $|\vec n|$, and satisfies the following multiple orthogonality properties
    \begin{proposition}
        Let $\vec n = (n_1,\dots,n_r)$ and $\vec j = (j_1,\dots,j_r)\in\NN_0^r$ be such that  $0\leq |\vec j| \leq |\vec n|/2$ and consider the polynomials $\mathbf P_{\vec j,\nu}^{\vec n}$, $\nu=1,\dots,\dim(\mathbb H_{|\vec n|-2|\vec j|}^d)$, defined in \eqref{eq:type-ii-radial}. Denote a $d$-variate monomial $\mathbf x^{\mathbf a}=x_1^{a_1}\cdots x_d^{a_d}$. Then
        \begin{equation}
        \label{eq:type-ii-orthogonality-radial}
            \int_{\mathbb R^2}\mathbf P_{\vec j,\nu}^{\vec n}(\mathbf x)\,\mathbf x^{\mathbf a} \,  W_l(\mathbf x)\, \dr \mathbf x = 0\qquad\text{ if }0\leq |\mathbf a|< |\vec n|-2|\vec j|+2j_l,
        \end{equation}
        for $1\leq l\leq r$ and $\nu=1,\dots,\dim(\mathbb H_{|\vec n|-2|\vec j|}^d)$.
    \end{proposition}
    \begin{proof}
        This proof is analogous to that of Proposition~\ref{prop:type-ii-orthogonality-d-vars}, with a modification in the integral concerning the radial part in the case $|\mathbf a|=|\vec n|-2|\vec j|+\ell$ with $\ell$ even, which reads
        $$
        I_1=\int_0^\infty p_{2\vec j}^{(2|\vec n|-4|\vec j|)}(\rho)\,\rho^\ell\,\rho^{2|\vec n|-4|\vec j|+1}\, \omega_l(\rho)\,\dr\rho.
        $$
        and by \eqref{eq:type-ii-orthogonality-radial-aux}, the result follows.
    \end{proof}
    
    Indeed, these properties keep an analogy with the ones stablished in Proposition~\ref{prop:type-ii-orthogonality-d-vars}.

    \subsection{Type~I MOP for radial weights}
    
    We now define the Type~I polynomials for radial weights

    \begin{definition} Let $\vec n = (n_1,\dots,n_r)$ and $\vec j = (j_1,\dots,j_r)\in\NN_0^r$ be such that  $0\leq |\vec j| \leq |\vec n|/2$ and consider the polynomials $q_{2\vec j,l}^{(2|\vec n|-4|\vec j|)}$, $1\leq l\leq r$ satisfying \eqref{eq:type-i-orthogonality-radial-aux}, and  $Y_\nu^{|\vec n|-2|\vec j|}$, $\nu=1,\dots,\dim(\mathbb H_{|\vec n|-2|\vec j|}^d)$ a basis of $\mathbb H_{|\vec n|-2|\vec j|}^d$. Then, we define
    \begin{equation}
        \label{eq:type-i-radial}
        \mathbf A^{\vec n}_{\vec j,\nu,(l)}(\mathbf x) = q^{(2|\vec n|-4|\vec j|)}_{2\vec j, l}(\rho)\, Y_\nu^{|\vec n|-2|\vec j|}(\mathbf x).
    \end{equation}
        
    \end{definition}

    Now, we study the Type~I multiple orthogonal properties satisfied by this functions.

    \begin{proposition} Let $\vec n = (n_1,\dots,n_r)$ and $\vec j = (j_1,\dots,j_r)\in\NN_0^r$ be such that  $0\leq |\vec j| \leq |\vec n|/2$ and consider the polynomials $\mathbf A_{\vec j,\nu,(l)}^{\vec n}$, $1\leq l\leq r$, $\nu=1,\dots,\dim(\mathbb H_{|\vec n|-2|\vec j|}^d)$, defined in \eqref{eq:type-i-radial}. Denote a $d$-variate monomial $\mathbf x^{\mathbf a}=x_1^{a_1}\cdots x_d^{a_d}$. Then,
    \begin{equation}
        \label{eq:type-ii-orthogonality-radial}
        \int_{\mathbb R^2}\sum_{l=1}^r\mathbf A_{\vec j,\nu,(l)}^{\vec n}(\mathbf x) \,\mathbf x^{\mathbf a}\, W_l(\mathbf x)\, \dr \mathbf x = 0\qquad\text{ if }0\leq|\mathbf a|<|\vec n|-1.
    \end{equation}
    \end{proposition}
    \begin{proof}
        The proof is parallel to that of Proposition~\ref{prop:type-i-orthogonality-d-vars}. However, in this case the integral depending on $\rho$ is 
        $$
        \int_{0}^\infty \sum_{l=1}^r q^{(2|\vec n|-4|\vec j|)}_{2\vec j, l}(\rho)\, \rho^\ell\,\rho^{2|\vec n|-4|\vec j|+1}\, \omega_l(\rho)\,\dr\rho,
        $$
        and \eqref{eq:type-i-orthogonality-radial-aux} leads to the result.
        
    \end{proof}

    As seen, the multiple orthogonal properties satisfied by these polynomials are quite similar to those in Propositions~\ref{prop:type-ii-orthogonality-d-vars} and \ref{prop:type-i-orthogonality-d-vars}.

    \subsection{Multiple Orthogonal Polynomials for multivariate Hermite weights}

    As an example, we consider, on the full space $\RR^d$, 
    $$
    W_l(x_1,\dots,x_d)=e^{-c_l(x_1^2+\cdots+x_d^2)}, \qquad \mathbf x =(x_1,\dots,x_d)\in\RR^d,\quad 1\leq l \leq r,
    $$
    where $c_1,\dots,c_r>0$ and $c_i\neq c_j$ if $i\neq j$. Using spherical polar coordinates \eqref{eq:sph-coordinates-d}, these weights reduce to 
    $$
    W_l(\rho, \theta_1,\dots,\theta_{d-1})=e^{-c_l \rho^2}, \qquad \rho\geq 0, \quad 1\leq l\leq r.
    $$
    In this way, $W_l$ is a radial weight such that $W_l(\mathbf x)=\omega_l(\rho)= e^{-c_l\rho^2}$, $1\leq l\leq r$. Multiple orthogonal polynomials with respect to this set of weights are given below. First, recall from Section~\ref{subsec:Laguerre-MOP} the multiple Laguerre polynomials of the second kind, which are multiple orthogonal with respect to the weights $t^\alpha e^{-c_l t}$ on $[0,\infty)$. For this reason, they provide the polynomials leading the radial part of our next definitions.

    \begin{definition}
        Let $\vec n = (n_1,\dots,n_r)$ and $\vec j = (j_1,\dots,j_r)\in\NN_0^r$ be such that  $0\leq |\vec j| \leq |\vec n|/2$, and consider the Type~II and Type~I Laguerre MOPs of the second kind $L^{(|\vec n|-2|\vec j|,\vec c)}_{\vec j}$ and $B^{(|\vec n|-2|\vec j|,\vec c)}_{\vec j,1},\dots,B^{(|\vec n|-2|\vec j|,\vec c)}_{\vec j,r}$, introduced in Section~\ref{subsec:Laguerre-MOP}. Let $\{Y_\nu^{|\vec n|-2|\vec j|}:1\leq\nu\leq\dim\mathbb H_{|\vec n|-2|\vec j|}^d\}$ be an orthogonal basis of spherical harmonics. Then, we define
        \begin{itemize}
            \item \textbf{Type II}:
            \begin{equation}\label{eq:type-ii-Laguerre-d-vars}
                \mathbf P^{\vec n}_{\vec j,\nu}(\mathbf x) = L^{(|\vec n|-2|\vec j|,\vec c)}_{\vec j}(\rho^2) Y_\nu^{|\vec n|-2|\vec j|}(\mathbf x),\qquad \mathbf x\in\RR^d.
            \end{equation}
            \item \textbf{Type I}:
            \begin{equation}\label{eq:type-i-Laguerre-d-vars}
                \mathbf B^{\vec n}_{\vec j,\nu,(l)}(\mathbf x) = B^{(|\vec n|-2|\vec j|,\vec c)}_{\vec j,l}(\rho^2) Y_\nu^{|\vec n|-2|\vec j|}(\mathbf x),\qquad \mathbf x\in\RR^d.
            \end{equation}
        \end{itemize}
    \end{definition}

    Indeed, these constructions are again polynomials in Cartesian coordinates $x_1,\dots,x_d$ and satisfy multiple orthogonality relations, which we make explicit in the next results. The proofs are omitted owing to their similarity to the previous ones, using the multiple orthogonal properties of multiple Laguerre polynomials of the second kind \eqref{eq:mop-Laguerre-type-ii-1-var} and \eqref{eq:mop-Laguerre-type-i-1-var} for the radial part.

    \begin{proposition}
        Let $\vec n = (n_1,\dots,n_r)$ and $\vec j = (j_1,\dots,j_r)\in\NN_0^r$ be such that  $0\leq |\vec j| \leq |\vec n|/2$, and consider the $d$-variate polynomials $\mathbf P^{\vec n}_{\vec j,\nu}$ introduced in \eqref{eq:type-ii-Laguerre-d-vars} and $\mathbf B^{\vec n}_{\vec j,\nu,(l)}$, $1\leq l\leq r$ in \eqref{eq:type-i-Laguerre-d-vars}, for $1\leq \nu\leq \dim\mathbb H_{|\vec n|-2|\vec j|}^d$. Denote by $\mathbf x^{\mathbf a}=x_1^{a_1}\cdots x_d^{a_d}$ a $d$-variate monomial. Then,
        \begin{itemize}
            \item \textbf{Type II}: for  $1\leq l \leq r$
            \begin{equation}
            \int_{\mathbb R^d} \mathbf P^{\vec n}_{\vec j,\nu}(\mathbf x)\, \mathbf x^{\mathbf a}\, W_l(\mathbf x)\, \dr\mathbf x = 0 \qquad\text{ if } |\mathbf a|<|\vec n|-2|\vec j|+2j_l.
            \end{equation}
            
            \item \textbf{Type I}:
            \begin{equation}
            \sum_{l=1}^r\int_{\mathbb R^d} \mathbf B^{\vec n}_{\vec j,\nu,(l)}(\mathbf x)\, \mathbf x^{\mathbf a}\, W_l(\mathbf x)\, \dr\mathbf x = 0 \qquad\text{ if } |\mathbf a|<|\vec n|-2.
            \end{equation}
            \end{itemize}
    \end{proposition}

    \section{Conclusions and future work}

    Drawing upon the established theory of orthogonal polynomials on the unit ball $B^d$, we use multiple orthogonal polynomials to extend this framework to the multivariate setting. Specifically, we preserve the classical structure from \cite{DX14}, where polynomials are defined as the product of a radial component and an angular component, the latter given by spherical harmonics. In our extension, the radial part, traditionally associated with Jacobi polynomials, is replaced by Jacobi-Pi\~neiro MOPs. This modification enables the construction of multivariate polynomials on the ball that satisfy Type~I and Type~II multiple orthogonality relations, thereby generalizing the univariate Jacobi-Pi\~neiro cases \eqref{eq:type-i-orthogonality-1-var} and \eqref{eq:type-ii-orthogonality-1-var}. 
    
    To illustrate the usefulness of this construction, we provide an extension of the Nearest Neighbor Relations using arbitrary paths and carry out a detailed study of the case in which the multi-indices follow the step-line. Finally, we address an extension of bivariate multiple orthogonality for radial weights, providing an example for Hermite-type weights where multiple Laguerre polynomials of the second kind were employed in the definitions.

    In future work, we aim to extend existing constructions for bivariate orthogonal polynomials—such as those presented in \cite[Section 2.6.1]{DX14}—to the multiple orthogonal setting. This will enable the generalization of various bases on planar domains, providing a broader framework for multivariate multiple orthogonality.
    
	\section*{Acknowledgements}

    The work is partially supported by grants PID2023-149117NB-I00, PID2024-155133NB-I00 and CEX
    2020-001105-M, all funded by ``Ministerio de Ciencia, Innovación y Universidades'' 
    
    \noindent(MICIU/AEI/10.13039/501100011033 and ERDF/EU), Spain.

	\bibliographystyle{plain}
    \bibliography{refs}{}

    \newpage
	
	\appendix

    \section{Appendix}\label{sec:appendix}

    \begin{lemma}
    \label{lemma:trig-integrals}
        Consider $a,b,m\in\NN_0$. Then,
        \begin{align}
        \label{eq:trig-integrals}
            \int_0^{2\pi}\cos(mt) \cos^a(t) \sin^b(t)\, \dr t &=0, & 
            \int_0^{2\pi}\sin(mt) \cos^a(t) \sin^b(t)\, \dr t &=0,
        \end{align}
        if one of the following conditions is satisfied:
        \begin{itemize}
            \item $a+b<m$, or
            \item $a+b = m+\ell$, with $\ell$ and odd natural number. 
        \end{itemize}
    \end{lemma}
    
    \begin{proof}
        Despite the fact that these are real valued integrals, we will use some useful results from complex analysis to prove this result. First, recall the equalities
        \begin{align}
            \cos(t) &= \dfrac{e^{it}+e^{-it}}{2}, &
            \sin(t) &= \dfrac{e^{it}-e^{-it}}{2 i}, 
        \end{align}
        So that 
        \begin{align*}
            \cos^a(t) &= \frac 1 {2^a}\sum_{j=0}^a\binom a j e^{it(a-2j)},&
            \sin^b(t) &= \frac 1 {(2i)^b}\sum_{h=0}^b (-1)^h\binom b h e^{it(b-2h)},\\
            \cos(mt) &= \dfrac{e^{imt}+e^{-imt}}{2},&
            \sin(mt) &= \dfrac{e^{imt}-e^{-imt}}{2 i}.
        \end{align*}
        Substituting these expressions in \eqref{eq:trig-integrals}, we have
        \begin{multline}
        \label{eq:lemma-linear-combination-cos}
        \int_0^{2\pi}\cos mt \cos^a t \sin^b t \dr\theta=\frac 1 {2^{a+b+1}i^b}\sum_{j=0}^a\sum_{h=0}^b(-1)^h\binom{a}{j}\binom{b}{h}\Big(\int_0^{2\pi}e^{it(a+b-2(j+h)+m)}\dr t\\+\int_0^{2\pi}e^{it(a+b-2(j+h)-m)}\dr t\Big),
        \end{multline}
        \begin{multline}
        \label{eq:lemma-linear-combination-sin}
        \int_0^{2\pi}\sin mt \cos^a t \sin^b t \dr\theta=\frac 1 {2^{a+b+1}i^{b+1}}\sum_{j=0}^a\sum_{h=0}^b(-1)^h\binom{a}{j}\binom{b}{h}\Big(\int_0^{2\pi}e^{it(a+b-2(j+h)+m)}\dr t\\-\int_0^{2\pi}e^{it(a+b-2(j+h)-m)}\dr t\Big),
        \end{multline}
        On the other hand, it is easy to check that
        \begin{equation*}
            \int_0^{2\pi}e^{itn}\dr t =\begin{cases}
                0 & \text{ if } n\neq0\\
                2 \pi & \text{ if }n=0.
            \end{cases}
        \end{equation*}
        
        Now, assume $a+b<m$, so that, if $0\leq j\leq a$, $0\leq h\leq b$, observe that
        $$
        a+b-2(j+h)+m\geq -a-b+m>0
        $$
        and
        $$
        a+b-2(j+h)-m\leq a+b-m<0,
        $$
        so that both integrals in the right-hand side of \eqref{eq:lemma-linear-combination-cos} and \eqref{eq:lemma-linear-combination-sin} vanish for every value of $j$ and $h$ and the result follows in this case.

        If $a+b=m+\ell$, with $\ell$ an odd integer, the integrals in the right-hand side of \eqref{eq:lemma-linear-combination-cos} and \eqref{eq:lemma-linear-combination-sin} do not vanish only if $(2m+\ell)/2= j+h$ and $\ell/2= j+h$ respectively. If $\ell$ is odd, these conditions will never be satisfied, so that both integrals will be zero for every value of $j$ and $h$, and the result follows too.  
    \end{proof}

\end{document}